\documentclass[10pt,oneside]{article}

\usepackage{natbib}
\setcitestyle{authoryear} 

\usepackage{lmodern}
\usepackage{times}
\usepackage{mathtools}

\usepackage{amssymb,amsmath,amsthm}
\usepackage{dsfont}
\newcommand{\ind}[1]{\text{1}\{#1\}}
\usepackage{bbold}
\usepackage[short]{optidef}

\usepackage{framed,multirow,multicol}
\usepackage{subcaption}

\usepackage{xcolor,xspace}
\usepackage{lscape}
\usepackage{graphicx,epsfig,tikz,caption}
\usepackage[normalem]{ulem}
\usepackage{enumerate}
\usepackage{verbatim}
\usepackage{makeidx,latexsym}
\usepackage[colorlinks=true,citecolor=blue]{hyperref}%

\usepackage{bm}
\usepackage{stmaryrd}
\usepackage{lscape}

\usepackage[english]{babel}
\usepackage{bbm}
\usepackage{pgfplots}
\usepackage[para]{footmisc}

\usepackage{parskip}
\usepackage[letterpaper,margin=1.1in]{geometry}

\renewcommand{\paragraph}[1]{\noindent\textbf{#1}\quad}
\newcommand{\compilehidecomments}{false} 
\ifthenelse{ \equal{\compilehidecomments}{true} }{%
	\newcommand{\yiliu}[1]{}
	\newcommand{\milan}[1]{}
}{
	\newcommand{\yiliu}[1]{{\color{red}  [Yiliu:#1]}}
	\newcommand{\milan}[1]{{\color{blue} [Milan:#1]}}
}

\usepackage[utf8]{inputenc} 
\usepackage[T1]{fontenc}    
\usepackage{hyperref}       
\usepackage{url}            
\usepackage{booktabs}       
\usepackage{nicefrac}       
\usepackage{microtype}      
\usepackage{xcolor}         

\usepackage{amsmath,bm}
\usepackage{bbm}
\usepackage{amsfonts}
\usepackage{amsthm}
\usepackage{algorithm}
\usepackage{algorithmic}

\usepackage{xcolor,xspace}
\usepackage{lscape}
\usepackage{graphicx,epsfig,tikz}
\usepackage[normalem]{ulem}
\usepackage{enumerate}
\usepackage{verbatim}
\usepackage{makeidx,latexsym}

\newtheorem{theorem}{Theorem}[section]
\newtheorem{corollary}{Corollary}[section]
\newtheorem{lemma}[theorem]{Lemma}

\newtheorem{definition}{Definition}[section]

\makeatletter
\def\munderbar#1{\underline{\sbox\tw@{$#1$}\dp\tw@\z@\box\tw@}}
\makeatother

\renewcommand{\Pr}{\mathbb{P}}
\newcommand{\E}{\mathbb{E}}

\title{Sketching stochastic valuation functions}

\date{}

\author{ Milan Vojnovi\' c \\
	Department of Statistics\\
	London School of Economics\\
	\texttt{m.vojnovic@lse.ac.uk} \\
	\And
	\hspace{1mm}Yiliu Wang \\
	Allen Institute\\
	Department of Neurobiology \& Biophysics, University of Washington\\
	\texttt{yiliuw@uw.edu} \\
}

\usepackage{caption}
\usepackage[labelfont=sf]{subcaption}
\RequirePackage{tgtermes}
\RequirePackage{newtxtext}
\RequirePackage{newtxmath}
\RequirePackage{bm}
\RequirePackage{endnotes}

\usepackage{algorithm}
\usepackage{algorithmic}
\usepackage{tikz}

\usepackage{natbib}
 \bibpunct[, ]{(}{)}{,}{a}{}{,}%

\renewcommand{\Pr}{\mathbb{P}}

\begin{document}

\title{Sketching Stochastic Valuation Functions}

\author{
Milan Vojnovi\'c%
\thanks{Department of Statistics, London School of Economics. Email: \texttt{m.vojnovic@lse.ac.uk}}
\and
Yiliu Wang%
\thanks{Allen Institute. Email: \texttt{yiliu.wang@alleninstitute.org}}
\thanks{Department of Neurobiology \& Biophysics, University of Washington. Email: \texttt{yiliuw@uw.edu}}
}

\maketitle

\begin{abstract}
We consider the problem of sketching set valuation functions, defined as the expectation of a valuation function applied to independent random item values. For valuation functions that are monotone and either subadditive or submodular, and that satisfy a weak homogeneity condition (or other structural conditions), we show that there exist discretized versions of the item value distributions---each with support size $O(k \log k)$---that yield a sketch valuation function providing a constant-factor approximation to the true valuation for any subset of items of size at most $k$. These discretized distributions can be computed efficiently for each item independently, making the approach highly scalable. Our results apply broadly to valuation functions commonly encountered in practice, including team performance based on the best member (e.g., maximum functions), constant elasticity of substitution (CES) production functions with diminishing returns in economics, and others. Sketch valuation functions are especially useful in optimization problems such as best set selection and welfare maximization, where exact value computations are costly or intricate. They enable efficient approximate evaluation of value oracle queries while preserving provable approximation guarantees for the original stochastic optimization problem.
\end{abstract}

\section{Introduction}\label{sec:Intro}
Evaluating the value of a set of items is critical in many applications, including ranking and selection tasks in recommender systems, information retrieval, assortment optimization, and team formation on online platforms such as gaming and freelancing. In e-commerce, information retrieval, and recommender systems, items like products, documents, or media are recommended based on predicted relevance scores. These scores—computed using machine learning models informed by item features, user preferences, and context—are inherently uncertain \cite{uncertaintyretrieval,risky}. In digital advertising, ads are selected based on uncertain click-through rate predictions \cite{ctr}. In online gaming and freelancing platforms, items correspond to players or workers, and their values represent individual skills or performance on specific tasks. Efficiently computing set valuations under uncertainty is essential, particularly when some approximation error is acceptable. This task is challenging due to both the uncertainty in individual item values and the non-linear nature of set valuation functions.

We consider the problem of approximating a class of stochastic valuation functions defined over a ground set of items $\Omega = \{1, \ldots, n\}$, where the function takes the form 
\begin{equation*}
u(S) = \mathbb{E}[f(X_S)] \quad \text{for } S \subseteq \Omega,
\label{equ:us0}
\end{equation*}
with $f: \mathbb{R}_+^n \rightarrow \mathbb{R}_+$ being a given function. The vector $X_S$ has its $i$-th component equal to $X_i$ if $i \in S$ and equal to a neutral element $0$ otherwise. The variables $X_1, \ldots, X_n$ represent the values of individual items, modeled as independent random variables with cumulative distribution functions $P_1, \ldots, P_n$, respectively. Our goal is to approximate $u$ using a sketch function $v$ such that for a given $\mathcal{F} \subseteq 2^{\Omega}$, 
\[
v(S) \leq u(S) \leq \alpha v(S), \quad \forall S \in \mathcal{F},
\] 
for some $\alpha \geq 1$. A function $v$ satisfying this condition is called an \emph{$\alpha$-approximate sketch valuation function} (or simply an \emph{$\alpha$-sketch}) on $\mathcal{F}$. When $\mathcal{F} = 2^{\Omega}$, $v$ is an $\alpha$-approximation of $u$ on all subsets. It is crucial that the sketch function be efficiently computable, enabling fast approximate value queries. We are particularly interested in applying such sketches to optimization problems, including best set selection and welfare maximization.

\begin{figure}
\begin{center}
\includegraphics[width=0.8\linewidth]{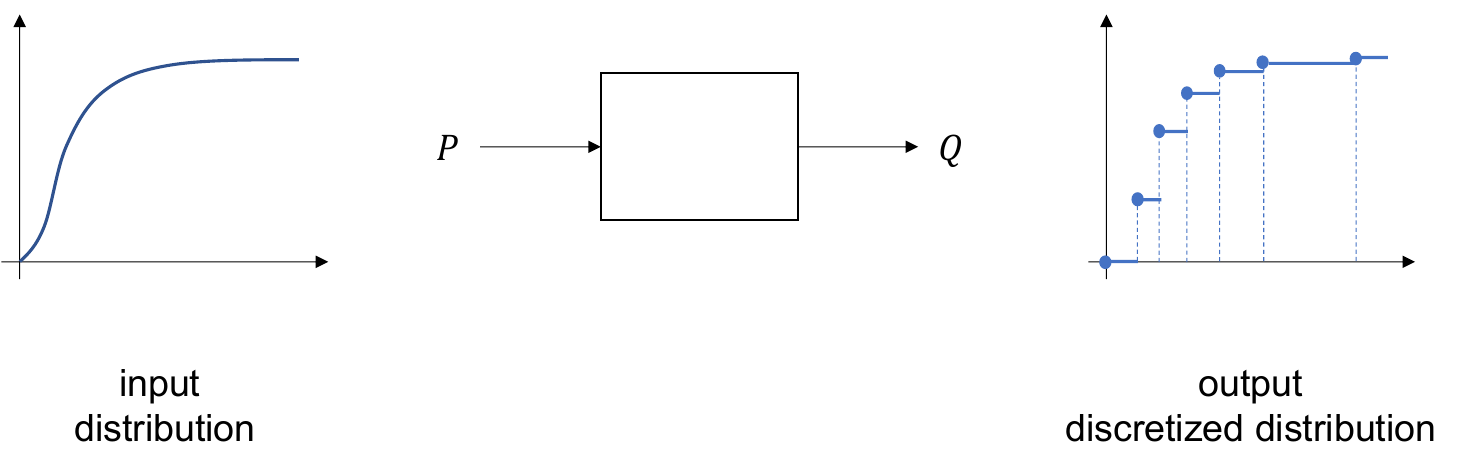} \caption{Sketching by discretizing distributions. Each item's value distribution is discretized independently, such that the resulting discrete distributions have finite supports and provide a satisfactory approximation guarantee when used to evaluate the corresponding sketch set valuation function.\label{fig:sketch}} 
\end{center}
\end{figure}

We propose a method for approximating stochastic valuation functions by discretizing the distributions of individual item values into finite supports, yielding a sketch function that guarantees a constant-factor approximation. This approach enables efficient evaluation of value queries for any subset of items up to a specified size. Specifically, our algorithm converts the cumulative distribution function of each item $P_i$ into a discrete cumulative distribution function $Q_i$ (see Figure~\ref{fig:sketch}). We establish approximation guarantees showing that, for monotone subadditive or submodular valuation functions satisfying a weak homogeneity condition (or certain other conditions), the algorithm produces discretized distributions with support size $O(k \log k)$, yielding a constant-factor approximation sketch for any query involving a set of up to $k$ items. We further show that using such sketch functions guarantees a constant-factor approximation for best set selection and welfare maximization problems, provided that $k$ is small enough relative to $n$.

\subsection{Related Work}

We begin with a brief review of theoretical results on the approximation limits for certain classes of set functions. The work of \cite{GHIM09} was the first to formulate the problem of approximating a submodular function \emph{everywhere}, i.e., approximating its value for all subsets of items given access to a value oracle. They presented an $O(\sqrt{n} \log n)$ approximation algorithm and proved that no algorithm can achieve a better approximation factor than $\Omega(\sqrt{n}/\log n)$. \cite{Balcan} showed that for some matroid rank functions—a subclass of submodular set functions—every sketch fails to achieve an approximation factor better than $n^{1/3}$. \cite{BDFKNR12} extended these results by showing that every subadditive set function $u$ has an $\tilde{O}(\sqrt{n})$ sketch, where the $\tilde{O}(\cdot)$ notation hides polylogarithmic factors in $n$, using a polynomial number of demand queries. They also demonstrated that any deterministic algorithm relying solely on a value oracle cannot guarantee an approximation factor better than $n^{1-\epsilon}$. The algorithms discussed above are based on geometric constructions, particularly in finding an ellipsoid that closely approximates the polymatroid associated with $u$.

Building on this foundational work, subsequent research has focused on improving computational and query complexity. \cite{CD17} introduced faster and simpler algorithms for sketching valuation functions. They presented an algorithm that computes a $\tilde{O}(\sqrt{n})$ sketch for submodular set functions using $\tilde{O}(n^{3/2})$ value queries, as well as another algorithm that achieves a $\tilde{O}(\sqrt{n})$ approximation for subadditive functions using $O(n)$ demand and value queries.

The problem of approximating the expected value of a function of independent random variables was first studied in \cite{K81}, which focused on approximating the expected value of certain functions of a sum of independent random variables using a function of the expected values of their univariate marginal distributions. \cite{AN16} addressed the problem of maximizing a monotone submodular function, defined as the expected value of a monotone submodular value function, subject to a matroid constraint. \cite{KR18} studied this problem under cardinality constraints, focusing on a class of test score algorithms that use one-dimensional representations of each item's value distribution. They showed that for objective functions defined by the sum of top-order statistics, certain test scores yield constant-factor approximations. Building on this, \cite{SVY21} proved within a sketch-based framework the existence of test scores that provide constant-factor approximations for monotone submodular functions satisfying an extended diminishing returns property. Notably, they identified an $O(\log k)$-approximate sketch function using a $k$-dimensional test score. To the best of our knowledge, this is the best-known sketch for monotone stochastic valuation functions satisfying the extended diminishing returns property. \cite{lee2021test} further extended the test score framework to address stochastic valuation maximization under more general budget constraints. 

\cite{Mehta} introduced a Polynomial Time-Approximation Scheme (PTAS) for the stochastic valuation maximization problem with a maximum value objective function and a cardinality constraint limiting the set size to at most $k$ items. Their approach represents each item's distribution using a histogram of size $O(k \log k)$. Our algorithm adopts a similar distribution discretization strategy, involving truncation of the tail and exponential binning of the remaining distribution. A key distinction is that our algorithm computes a discrete distribution for each item independently, whereas \cite{Mehta} used a common set of binning boundaries for all items' distributions. Another difference to our work is in that we provide sketching results for approximating a stochastic valuation function. 

\subsection{Organization of the Paper}

In Section~\ref{sec:problem}, we formally define the problem and present the necessary background and preliminary results. Section~\ref{sec:algoapprox} introduces our sketching algorithm and establishes approximation guarantees for set functions under various assumptions. In Section~\ref{sec:opt}, we discuss the implications of our sketch functions for approximately solving best set selection and welfare maximisation problems. Numerical results are presented in Section~\ref{sec:num}. Finally, concluding remarks are provided in Section~\ref{sec:conc}. Proofs omitted from the main text are available in the online companion.

\section{Problem Formulation}
\label{sec:problem}

Let $\Omega = \{1, \ldots, n\}$ be a ground set of items. Each item $i \in \Omega$ has a value according to the random variable $X_i$ with distribution $P_i$. The variables $X_1, \ldots, X_n$ are assumed to be independent. The value of a set of items $S \subseteq \Omega$ is given by the set valuation function $u(S)$ defined as
\begin{equation}
u(S) = \mathbb{E}[f(X_S)],
\label{equ:sutil}
\end{equation}
where $f: \mathbb{R}^n \rightarrow \mathbb{R}_+$ is a monotone function. Hereinafter, for any set $S \subseteq \Omega$ and any vector $x=(x_1, \ldots, x_n)^\top$, we use the notation $x_S$ to denote a vector $z \in \mathbb{R}^n$ such that $z_i = x_i$ if $i \in S$ and $z_i = 0$ otherwise.

We compute $Q_1, \ldots, Q_n$ as representations of $P_1, \ldots, P_n$ corresponding to the distributions of discrete random variables $Y_1, \ldots, Y_n$. We define the sketch set function $v$ as the expected value of $f$ with respect to the item value distributions $Q_1, \ldots, Q_n$. The sketch function $v$ should approximate the set valuation function $u$ within a multiplicative error tolerance on a given $\mathcal{F} \subseteq 2^\Omega$, i.e., for some $\alpha \geq 1$,
\begin{equation*}
v(S) \leq u(S) \leq \alpha v(S), \quad \forall S \in \mathcal{F}.
\end{equation*}
When this guarantee holds, we say that $v$ is an \emph{$\alpha$-approximation} of $u$, or that $v$ is \emph{an $\alpha$-sketch} of $u$, on $\mathcal{F}$. If $\mathcal{F} = 2^\Omega$, then $v$ is an $\alpha$-approximation of $u$ \emph{everywhere}. 

For a positive integer $k$, let $\mathcal{F}_k = \{S \subseteq \Omega : |S| \leq k\}$. The problem we consider is to design an algorithm that outputs an $\alpha$-sketch valuation function on $\mathcal{F}_k$, for any given $k \in [n]$. This is considered in the computation model under which the algorithm has access to a value oracle that outputs the value $P_i(x)$ for any input $x \in \mathbb{R}$ and $i \in \Omega$. Preferably, a good sketch has a small representation size, and the approximation factor $\alpha$ should not depend on $k$. We aim to address the challenge of finding good sketches for a wide range of function classes, defined as the expected value with respect to discretized item value distributions.

\subsection{Background on Function Classes and Their Properties} 
In what follows, we define several classes of valuation functions and summarize key properties used throughout the paper. Readers familiar with these concepts may choose to skip this section.

A function $f$ is \emph{monotone} if $f(x) \leq f(y)$ for every $x, y$ in the domain of $f$ such that $x \leq y$, where the inequality is interpreted coordinate-wise.

\subsubsection{Diminishing Returns and Submodular Functions}

A function $f$ is \emph{submodular} if it satisfies a diminishing returns property over a product domain $\mathcal{X} = \mathcal{X}_1 \times \cdots \times \mathcal{X}_n$, where each $\mathcal{X}_i \subset \mathbb{R}$ is compact. Specifically, $f$ is submodular if for all $x, y \in \mathcal{X}$,
\[
f(x \wedge y) + f(x \vee y) \leq f(x) + f(y),
\]
where $\wedge$ and $\vee$ denote the coordinate-wise minimum and maximum, respectively. When $\mathcal{X} = \{0,1\}^n$, this definition coincides with the standard notion of submodularity for set functions. A twice-differentiable function $f$ is submodular if and only if all off-diagonal entries of its Hessian are non-positive, i.e., $\partial^2 f(x)/\partial x_i \partial x_j \leq 0$ for all $i \neq j$, for every $x$ \cite{topkis}.

A function $f$ satisfies the \emph{weak diminishing returns (weak DR)} property if, for every $x, y \in \mathcal{X}$ such that $x \leq y$ coordinate-wise, for every $i \in [n]$ such that $x_i = y_i$, and for every $z \in \mathbb{R}_+$ such that $x + z e_i \in \mathcal{X}$ and $y + z e_i \in \mathcal{X}$, the following holds:
\[
f(x + z e_i) - f(x) \geq f(y + z e_i) - f(y),
\]
where $e_i$ is the $i$-th standard basis vector. As shown in \cite{bian}, a function $f$ is submodular if and only if it satisfies the weak DR property.

A notable subclass of submodular functions is the class of DR-submodular functions (where ``DR'' stands for diminishing returns) \cite{bian,Soma}. A function $f$ is \emph{DR-submodular} if, for all $x, y \in \mathcal{X}$ such that $x \leq y$ coordinate-wise, any $i \in [n]$, and any non-negative scalar $z$ such that both $x + z e_i \in \mathcal{X}$ and $y + z e_i \in \mathcal{X}$, the following holds:
\[
f(x + z e_i) - f(x) \geq f(y + z e_i) - f(y).
\]
This property strengthens the weak DR condition and ensures that marginal gains diminish as the input grows.

The following lemma relates a function $f$ on a product domain to a set function $u$ defined as $u(S) = \mathbb{E}[f(X_S)]$ for independent random variables $X_1, \ldots, X_n$.

\begin{lemma}[Lemma~3 \cite{AN16}] 
Assuming that $f$ is a monotone submodular function, it follows that $u$ is a monotone submodular set function.
\label{lem:AN}
\end{lemma}

\subsubsection{Concave and Coordinate-Concave Functions}

A function $f$ is \emph{convex} on a domain $\mathcal{X}$ if, for all $x, y \in \mathcal{X}$ and $\lambda \in [0,1]$, 
\[
f(\lambda x + (1-\lambda)y) \leq \lambda f(x) + (1-\lambda) f(y).
\] 
The function is \emph{concave} if the inequality is reversed. Submodular functions may be concave, convex, or neither.

A function $f$ is \emph{coordinate-wise concave} if, for every $x \in \mathcal{X}$, index $i \in [n]$, and non-negative scalars $u, v$ such that $x + u e_i \in \mathcal{X}$, $x + v e_i \in \mathcal{X}$, and $x + (u + v) e_i \in \mathcal{X}$, it holds that
\[
f(x + u e_i) - f(x) \geq f(x + (u + v) e_i) - f(x + v e_i).
\]
If $f$ is twice-differentiable, coordinate-wise concavity is equivalent to $\partial^2 f(x)/\partial x_i^2 \leq 0$ for all $i \in [n]$ \cite{bian}. In the smooth setting, coordinate-wise concavity coincides with standard concavity along each coordinate. By \cite{bian}, $f$ is DR-submodular if and only if it is both submodular and coordinate-wise concave.

\subsubsection{Subadditive Functions}

A function $f$ is \emph{subadditive} if $f(x + y) \leq f(x) + f(y)$ for all $x$ and $y$ in its domain. Similarly, a set function $u$ is subadditive if $u(S \cup T) \leq u(S) + u(T)$ for all $S, T \subseteq \Omega$. Any non-negative submodular set function is also subadditive. Moreover, any DR-submodular function satisfying mild additional conditions is subadditive, as formalized below.

\begin{lemma}\label{lem:fact1} 
If $f$ is a DR-submodular function on $\mathcal{X} \subseteq \mathbb{R}_+^n$ with $0 \in \mathcal{X}$ and $f(0) \geq 0$, then $f$ is subadditive on $\mathcal{X}$.
\end{lemma}

Finally, we state a sufficient condition for the expected utility to be monotone and subadditive. 

\begin{lemma} 
Assume that $f$ is a monotone function that is either subadditive or submodular. Then, the set function $u$ is monotone and subadditive.
\label{lem:sub}
\end{lemma}

\section{Algorithm and Approximation Guarantees}
\label{sec:algoapprox}

In this section, we present our algorithm for approximating an item's value distribution with a discrete distribution supported on a finite set, and we provide approximation guarantees for sketch valuation functions defined as expectations with respect to these discrete distributions. 

We assume that the valuation function $f$ and the item value distributions $P_1, \ldots, P_n$ satisfy the condition that $\mathbb{E}[f(X_i) \mid X_i > \tau]$ is finite for all $i \in \Omega$ and $\tau \in \mathbb{R}_+$ such that $P_i(\tau) < 1$. This condition is used to approximate the tail of an item's value distribution. For valuation functions where $f(x_i) = x_i$, this is equivalent to requiring that $\mathbb{E}[X_i]$ is finite for all $i \in \Omega$. Examples include the maximum value function $f(x) = \max\{x_1, \ldots, x_n\}$ and the constant elasticity of substitution (CES) function $f(x) = (x_1^r + \cdots + x_n^r)^{1/r}$ for $r > 0$.

\subsection{Algorithm}
\label{sec:algo}

\begin{algorithm}[t] 
	\caption{Distribution Discretization Algorithm}
	\begin{algorithmic}
        \REQUIRE $\epsilon \in (0,1]$, $a \in (0,1)$
	\STATE $\tau \leftarrow \inf\{x \in \mathbb{R} : P(x) \geq 1 - \epsilon\}$
        \STATE $J \leftarrow \lceil \log_{1/(1-\epsilon)}(1/a) \rceil$
        \STATE $Q(0) = P(a\tau)$
        \FOR{$j = 1, \ldots, J$}
            \IF{$j < J$} 
                \STATE $Q(a\tau/(1-\epsilon)^{j-1}) = P(a\tau/(1-\epsilon)^j)$
            \ELSE
                \STATE $Q(a\tau/(1-\epsilon)^{J-1}) = P(\tau)$
            \ENDIF
		\ENDFOR
        \IF{$P(\tau) < 1$}
            \STATE $H \leftarrow \mathbb{E}[f(X) \mid X > \tau]$
            \STATE $Q(f^{-1}(H)) = 1$
        \ENDIF
	\end{algorithmic}
 \label{alg:disc}
\end{algorithm}

Algorithm~\ref{alg:disc} outputs a discrete distribution $Q$ with finite support given value oracle access to a distribution $P$. For clarity, the output distribution $Q$ is specified at all its jump points, is constant between jumps, and is right-continuous.

The algorithm uses a binning scheme: values below a lower threshold are assigned to zero, values above an upper threshold are mapped to a fixed value, and the values in between are partitioned into bins of exponentially increasing width. Two parameters, $a \in (0,1)$ and $\epsilon \in (0,1]$, control the bin structure and balance approximation accuracy against computational efficiency. The parameter $a$ sets the lower bound of the bins: larger $a$ reduces the support size but loosens the upper bound on the approximation. The parameter $\epsilon$ governs the overall truncation of the binning scheme: smaller $\epsilon$ yields finer discretization but increases the support size.

The lower threshold is set to $a\tau$ and the upper threshold to $\tau$, where $\tau$ is the $(1-\epsilon)$-quantile of $P$. The bins, starting at $a\tau$, grow by a factor of $1/(1-\epsilon)$. Let $J$ be the smallest integer such that $a\tau / (1-\epsilon)^J \geq \tau$. The bin boundaries are then $x_j = a\tau/(1-\epsilon)^{j-1}$ for $j=1,\ldots,J$, with $x_{J+1} = \tau$. The distribution $Q$ reallocates the mass of $P$ as follows:
\begin{itemize}
    \item All mass on $[0, a\tau]$ is moved to $0$.
    \item All mass on $(\tau, \infty)$ is assigned to $f^{-1}(H)$, where $H = \mathbb{E}[f(X) \mid X > \tau]$, if $P(\tau) < 1$.
    \item For each $j=1,\ldots,J$, all mass on $(x_j, x_{j+1}]$ is transferred to $x_j$.
\end{itemize}
See Figure~\ref{fig:algoutput} for an illustration.

The support of $Q$ is
\[
\mathcal{Q} = \{0\} \cup \left\{a\tau, \frac{a\tau}{1-\epsilon}, \ldots, \frac{a\tau}{(1-\epsilon)^{J-1}}\right\} \cup \mathcal{Q}^*,
\]
where $\mathcal{Q}^* = \{f^{-1}(H)\}$ if $P(\tau) < 1$, and $\mathcal{Q}^* = \emptyset$ otherwise. The support size satisfies
\[
s = O\left(\frac{1}{\epsilon} \log(1/a)\right).
\]

\begin{figure}[htbp]
\centering

\subcaptionbox{Input distribution $P$\label{fig:algoutputp}}{
\includegraphics[width=0.35\textwidth]{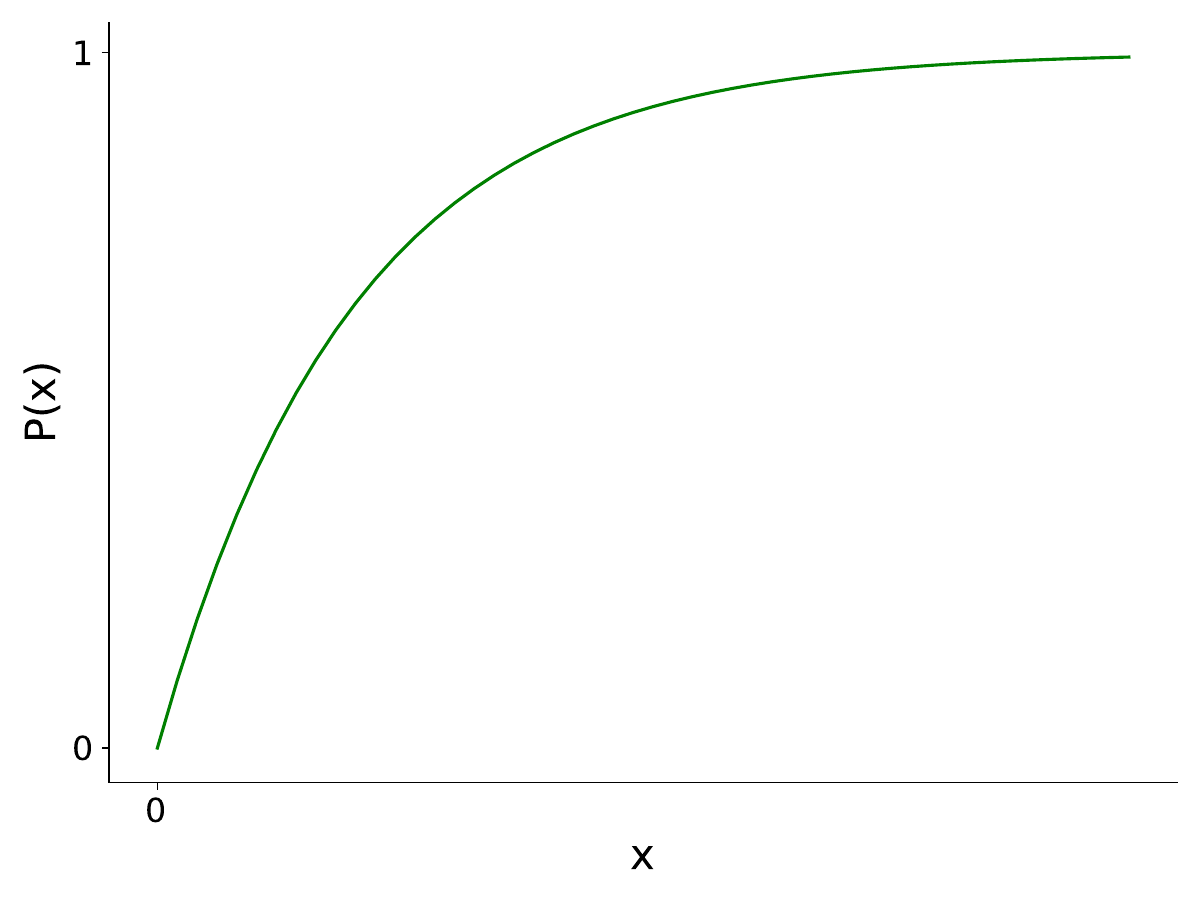}
}
\hspace{1.5cm}
\subcaptionbox{Output discretized distribution $Q$\label{fig:algoutputq}}{
\includegraphics[width=0.35\textwidth]{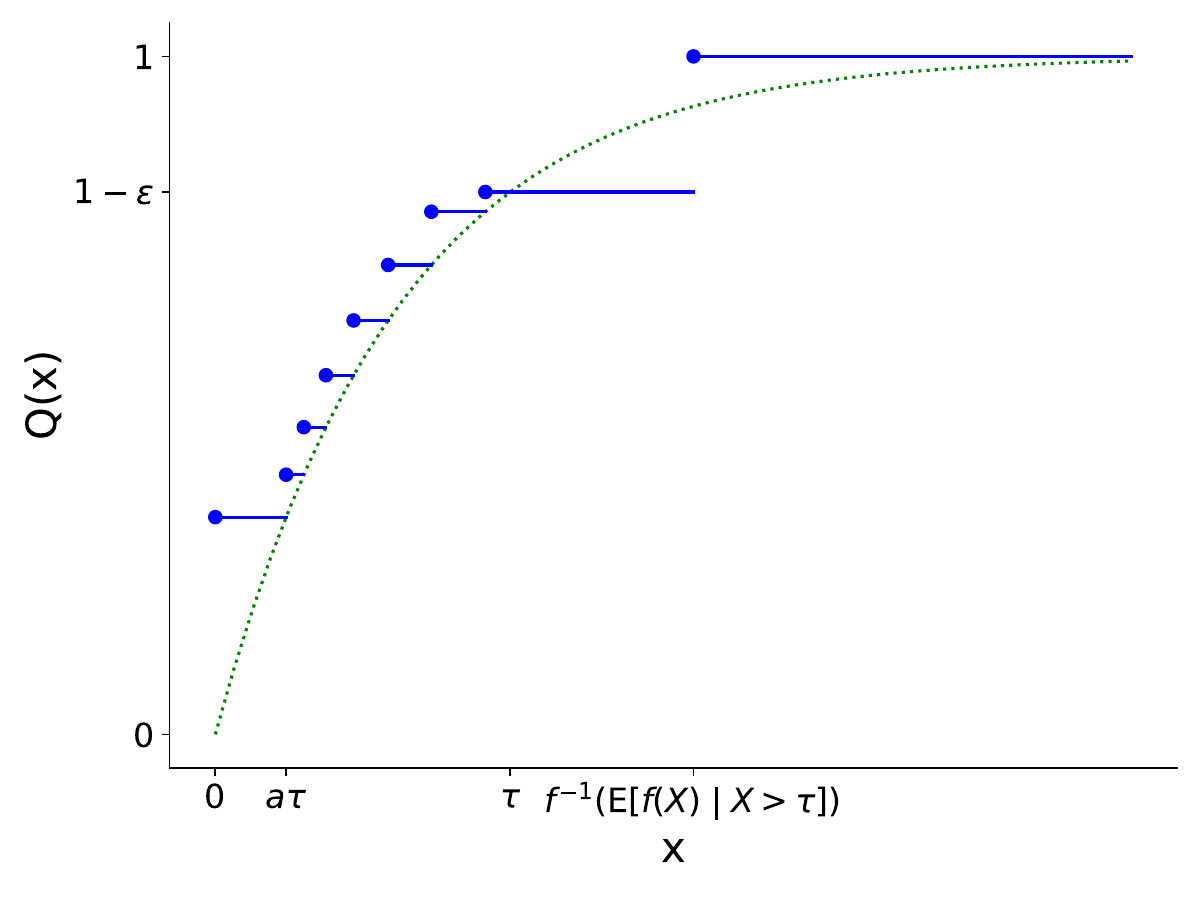}
}

\caption{Illustration of the distribution discretization by Algorithm~\ref{alg:disc}.}
\label{fig:algoutput}

\end{figure}

The algorithm guaranties the following relationships between $P$ and $Q$: 
(i) $Q(x) \geq P(x) - \epsilon$ for all $x$, 
(ii) $Q(x) \geq P(x)$ for all $x \leq \tau$, and 
(iii) $Q(x) \leq P(x) + \epsilon$ for all $x \geq \tau$. These properties are straightforward to verify; detailed proofs are provided in the online companion.

The algorithm requires computing: 
(P1) values of $P$ at $O((1/\epsilon)\log(1/a))$ points, 
(P2) the conditional expectation $\E[f(X)\mid X>\tau]$ if $P(\tau)<1$, and 
(P3) the $(1-\epsilon)$-quantile $\tau$. If $P$ is a discrete distribution with finite size support $m$, evaluating $P(x)$ for all $x$ takes $O(m)$ time, as does computing $\E[f(X)\mid X>\tau]$. Computing $\tau$ can be done in $O(m \log m)$ time by sorting the elements and selecting the appropriate rank. Importantly, these computations are performed only once per item, whereas directly evaluating $u(S)$ for any set $S$ of size at most $k$ requires $O(m^k)$ time, which becomes prohibitive even for small $k$ when $m$ is large.

\subsection{Approximation Guarantees}
\label{sec:approx}

We establish approximation guaranties for the set function 
$$
u(S) = \mathbb{E}[f(X_S)],
$$ 
using a sketch set function 
$$
v(S) = \mathbb{E}[f(Y_S)].
$$ 
Here, $X_1,\ldots, X_n$ are independent random variables with distributions $P_1, \ldots, P_n$, and $Y_1, \ldots, Y_n$ are independent random variables with distributions $Q_1,\ldots, Q_n$, where each $Q_i$ is obtained by applying Algorithm~\ref{alg:disc} to $P_i$. 

We assume that the distributions $P_1, \ldots, P_n$ have all atoms (if any) of mass at most $\Delta \in [0,1]$, i.e.,
\begin{equation}
P_i(x) - \lim_{z \uparrow x} P_i(z) \leq \Delta, \quad \forall x \in \mathbb{R}, \ \forall i \in \Omega.
\label{equ:picond}
\end{equation}
In fact, it suffices for \eqref{equ:picond} to hold only at $x = \tau_i$, the $(1-\epsilon)$-quantile of $P_i$. If each $P_i$ is continuous and strictly increases on its support, then $\Delta = 0$.  

\subsubsection{Weakly Homogeneous Functions} 
\label{sec:weak}

We consider functions that satisfy a weak homogeneity condition. Recall that $f$ is homogeneous of degree $d$ over a set $\Theta \subseteq \mathbb{R}$ if $f(\theta x) = \theta^d f(x)$ for all $x$ and $\theta \in \Theta$. We use a relaxed version:

\begin{definition} 
A function $f$ is \emph{weakly homogeneous of degree $d$ with tolerance $\eta$ over a set $\Theta \subseteq \mathbb{R}$} if
$$
\frac{1}{\eta}\, \theta f(x) \leq f(\theta x) \leq \theta^d f(x), \quad \forall x, \forall \theta \in \Theta.
$$
\label{def:wh}
\end{definition}

\begin{table}
\centering
\caption{Properties of some functions $f$. Note $x_{(i)}$ denotes the $i$th largest element of $x_1,\ldots,x_n$.}
\label{tab:fprop}
\begin{tabular}{c c c c c c c}
\hline
$f(x)$
& subadditive
& submodular
& convex
& concave
& $d$
& $\eta$ \\
\hline
$\max\{x_1,\ldots,x_n\}$
& \checkmark
& \checkmark
& \checkmark
&
& $1$
& $1$ \\
$x_{(1)}+\cdots+x_{(h)}$
& \checkmark
& \checkmark
& \checkmark
&
& $1$
& $1$ \\
$(\sum_{i=1}^n x_i^r)^{1/r}$, $r\ge 1$
& \checkmark
& \checkmark
& \checkmark
&
& $1$
& $1$ \\
$g(\sum_{i=1}^n x_i)$, concave $g$
& \checkmark
& \checkmark
&
& \checkmark
& min elasticity of $g$
& $1$ \\
$1-\prod_{i=1}^n(1-x_i)$
& \checkmark
& \checkmark
&
&
& $\le 1/2$, $n\ge 2$
& $1$ \\
\hline
\end{tabular}
\end{table}

Many functions are weakly homogeneous with positive degree and $\eta = 1$. Table~\ref{tab:fprop} provides examples. For a differentiable function $g:\mathbb{R}\to \mathbb{R}$, the elasticity at $z$ is $zg'(z)/g(z)$. 

\begin{theorem} 
\label{DR} 
Assume $f$ is monotone, subadditive or submodular, and weakly homogeneous with degree $d$ and tolerance $\eta$ over $[0,1]$, and let $\epsilon \in (\Delta,1)$. Then, for every set $S \subseteq \Omega$ with $|S| \leq k$,
$$
\frac{1}{2}(1-\epsilon)^{k-1}(1-\Delta/\epsilon) v(S) \leq u(S) \leq 2\eta \frac{1 + a^d k / \epsilon - \Delta/\epsilon}{(1-\epsilon)^k (1-\Delta/\epsilon)^2} v(S).
$$
\end{theorem}

The approximation factors depend on $d$, $\eta$, the discretization parameters $a$ and $\epsilon$, the set size $k$, and $\Delta$. The lower bound holds for any monotone subadditive or submodular $f$, independent of $d$ and $\eta$, and arises from bounding the upper ends of item value distributions. The lower and middle parts of the distribution are transformed monotonically, making the discretized distribution stochastically smaller up to the upper threshold. 

The factors $(1/2)(1-\epsilon)^{k-1}$ and $2/(1-\epsilon)^{k-1}$ come from upper-end transformations. The term $(1 + a^d k - \Delta/\epsilon)/(1-\Delta/\epsilon)$ arises from lower-end transformations, and $\eta/(1-\Delta/\epsilon)$ from the middle part. The upper bound increases with $d$ and $\eta$. 

\begin{corollary} 
\label{cor:cf}
Assume $d>0$ and $\Delta k < 1$. Under Theorem~\ref{DR}'s conditions, set $a = \epsilon^{2/d}$ and $\epsilon = c/k$ for $c \in (\Delta k, 1)$. Then, for all $S \subseteq \Omega$ with $|S| \leq k$,
$$
\frac{1}{2} e^{-\frac{c}{1-c}} (1-\Delta k/c) v(S) \leq u(S) \leq 2\eta e^{\frac{c}{1-c}} \frac{1+c}{(1-\Delta k/c)^2} v(S).
$$
\end{corollary}

Algorithm~\ref{alg:disc} produces an $\alpha$-sketch on ${\mathcal F}_k$ with
$$
\alpha = 4\eta \frac{1+c}{(1-\Delta k/c)^3} e^{2c/(1-c)}.
$$

When $\Delta = 0$, $\alpha$ can be arbitrarily close to $4\eta$ by choosing $c$ sufficiently small. Figure~\ref{fig:approx} illustrates $c \mapsto (1+c) e^{2c/(1-c)}$. If $\Delta = o(1/k)$, each discretized distribution has at most a support size
$$
s = O\left(\frac{1}{d} k \log k \right),
$$
which is $O(k \log k)$ if $d$ is bounded below by a positive constant.

\begin{figure}[htbp]
\centering
\includegraphics[width=0.35\linewidth]{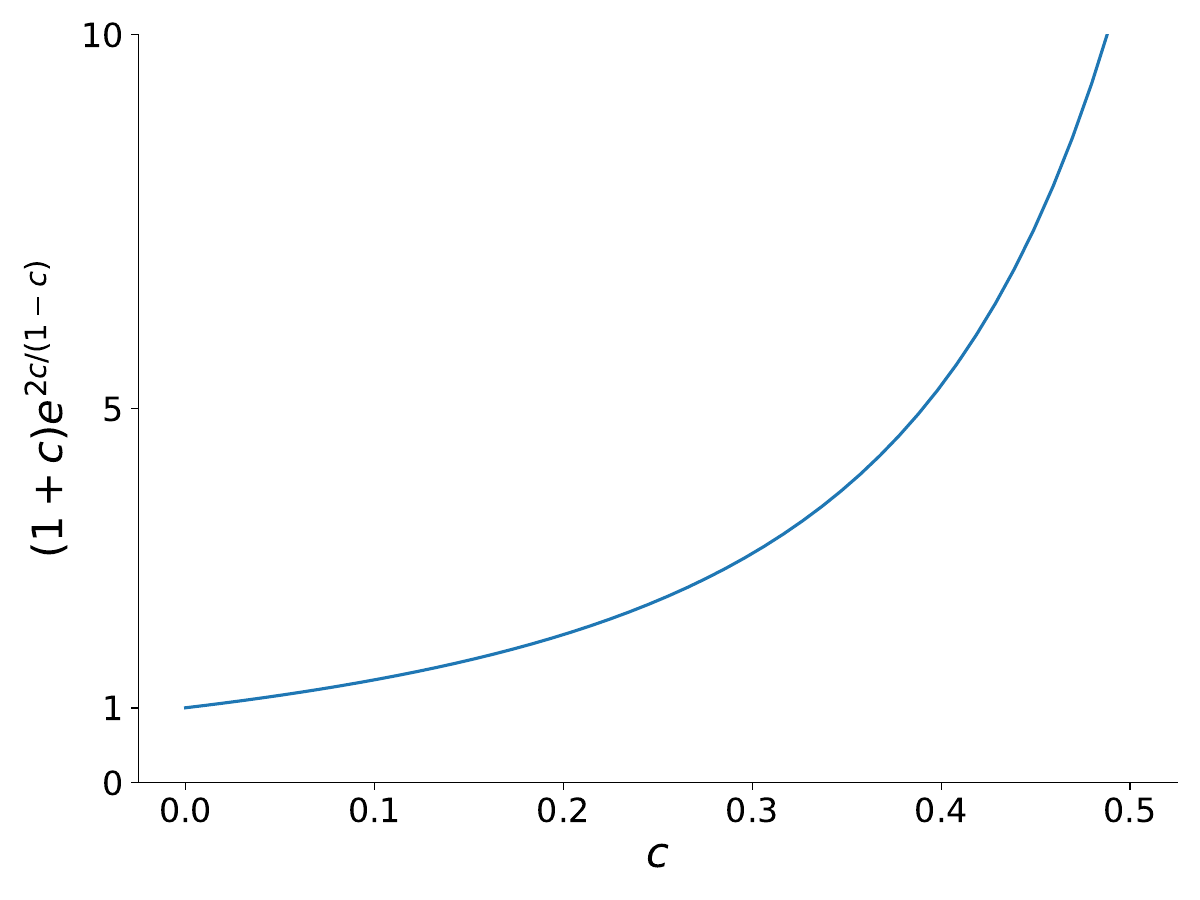}
\caption{Approximation factor $(1+c)e^{2c/(1-c)}$ versus $c$.}
\label{fig:approx}
\end{figure}

\paragraph{Proof of Theorem~\ref{DR}}
We outline the main steps of the proof, leaving the proofs of some lemmas to the online companion. Recall that Algorithm~\ref{alg:disc} performs three transformations: limiting the upper end of the support, limiting the lower end of the support, and applying exponential binning to the middle part. The proof proceeds through these steps and establishes the effect of each on the approximation factors.

We first introduce the notation used to define the aforementioned transformations. Let $\tau_i$ denote the $(1-\epsilon)$-quantile of distribution $P_i$, i.e., 
$$
\tau_i = \inf\{x\in \mathbb{R}: P_i(x)\geq 1-\epsilon\},
$$
where $\epsilon \in (0,1]$. 

The algorithm limits the upper end of the support of each distribution $P_i$. Define 
$$
H_i = 
\begin{cases} 
\E[f(X_i) \mid X_i>\tau_i], & \text{if } P_i(\tau_i) < 1\\
0, & \text{otherwise}
\end{cases}.
$$
Let $\hat{X}_i$ be a random variable equal to $X_i$ if $X_i\leq \tau_i$, and equal to $f^{-1}(H_i)$ otherwise. Here, $f^{-1}$ denotes the inverse of $f(x,0,\ldots,0)$ with respect to $x$. The support of $\hat{X}_i$ is contained in 
$$
[0, \tau_i] \cup \{f^{-1}(H_i)\}.
$$

The algorithm limits the lower end of the support by mapping values $\hat{X}_i \leq a \tau_i$ to $0$. Define 
$$
\tilde{X}_i = 
\begin{cases} 
\hat{X}_i, & \text{if } \hat{X}_i > a \tau_i\\
0, & \text{otherwise}.
\end{cases}.
$$
Then $\tilde{X}_i$ has support 
$$
\{0\} \cup [a \tau_i, \tau_i] \cup \{f^{-1}(H_i)\}.
$$

Each $\tilde{X}_i$ is then transformed into $Y_i$ by exponential binning over $[a\tau_i, \tau_i]$. Let $J=\lceil \log_{1/(1-\epsilon)}(1/a)\rceil$. Define the quantization function
$$
q(x; \tau, \epsilon, a) = \frac{1}{(1-\epsilon)^{j-1}}a\tau, \quad \text{for } x\in I_j(\tau,\epsilon,a), 1\leq j \leq J
$$
and $q(x;\tau,\epsilon,a)=x$ for $x\geq \tau$, where
$$
I_j(\tau,\epsilon,a) = \left(\frac{1}{(1-\epsilon)^{j-1}}a\tau,\frac{1}{(1-\epsilon)^{j}}a\tau\right], \quad 1\leq j < J,
$$
and
$$
I_J(\tau,\epsilon,a) = \left(\frac{1}{(1-\epsilon)^{J-1}}a\tau,\tau\right].
$$
Finally, define 
$$
Y_i = q(\tilde{X}_i; \tau_i, \epsilon, a).
$$

Assuming that $P_i(x)-\lim_{z\uparrow x} P_i(z)\leq \Delta$ for all $x\in \mathbb{R}$ and $i \in \Omega$, we have
$$
1-\epsilon \leq \Pr(X_i \leq \tau_i) \leq 1-\epsilon + \Delta.
\label{equ:pi}
$$

For the upper end, each item is divided into components depending on whether $X_i>\tau_i$. If $P_i(\tau_i) < 1$, then by the definition of $\hat{X}_i$,
$$
\E[f(\hat{X}_i)\mid \hat{X}_i>\tau_i] = \E[f(X_i) \mid X_i>\tau_i].
$$

Let $X_i^- := X_i \ind{X_i \le \tau_i}, \quad i \in [n].$, and
$$
w(S) = \E\left[f\left(X_S^-\right)\right].
$$

\begin{lemma} \label{upper-end-f}
Assume $f$ is monotone and subadditive or submodular. For every $S\subseteq \Omega$ with $|S| \leq k$,
$$
u(S) \geq (1-\epsilon)^{k-1} (1-\Delta/\epsilon) \max\left\{ \epsilon \sum_{i \in S} H_i, w(S) \right\},
$$
and
$$
u(S) \leq 2 \max\left\{ \epsilon \sum_{i \in S} H_i, w(S) \right\}.
$$
\end{lemma}

The proof is in the online companion. Using this, we compare 
$$
v_1(S) = \E[f(\hat{X}_S)] \quad \text{with} \quad u(S) = \E[f(X_S)].
$$

\begin{lemma}\label{uv1}
Assume $f$ is monotone and subadditive or submodular. Then, for $S\subseteq \Omega$ with $|S|\leq k$,
$$
\frac{1}{2}(1-\epsilon)^{k-1} (1-\Delta/\epsilon) v_1(S) \leq u(S) \leq 2 \frac{1}{(1-\epsilon)^{k-1}} (1-\Delta/\epsilon)^{-1} v_1(S).
$$
\end{lemma}

For the lower end, let 
$$
v_2(S) = \E[f(\tilde{X}_S)].
$$

\begin{lemma} 
Assume $f$ is monotone, subadditive or submodular, and weakly homogeneous in degree $d$ over $[0,1]$. For $S\subseteq \Omega$ with $|S|\leq k$,
$$
v_2(S) \geq \frac{1}{1 + a^d k/(\epsilon - \Delta)} v_1(S).
$$
\label{lem:v2v1ak}
\end{lemma}

For the middle part, each bin maps points to at least a $(1-\epsilon)$ fraction of the original value. Thus, the multiplicative loss is at most $(1-\epsilon)$. Let 
$$
v(S) = \E[f(Y_S)].
$$

\begin{lemma}\label{middle} 
Assume $f$ is monotone and weakly homogeneous with tolerance $\eta$. Then,
$$
v(S) \leq \frac{1-\epsilon}{\eta} v_2(S).
$$
\end{lemma}

Theorem~\ref{DR} follows by combining these results: the lower bound from Lemma~\ref{uv1}, and the upper bound from Lemmas~\ref{lem:v2v1ak} and \ref{middle}.

\paragraph{Approximation results for distributions with arbitrary point masses} The result in Theorem~\ref{DR} is derived under the assumption that, for every item value distribution, any atom (if present) has mass at most $\Delta$, with $\Delta < \epsilon < 1$. Corollary~\ref{cor:cf} holds under the same assumption and further requires $\Delta < c/k$ for some positive constant $c$. Moreover, a constant-factor approximation guarantee is achieved when $\Delta = o(1/k)$. Hence, these results are valid only when $\Delta$ is sufficiently small relative to either $\epsilon$ or $1/k$. This restriction can, however, be alleviated by redefining the sketch valuation function as follows. 

For any given $\Delta \in (0,1)$, each item value distribution can be represented as a mixture of distributions: one in which every atom (if any) has mass at most $\Delta$, and the others being point mass distributions. The number of such point mass components is at most $1/\Delta$. This decomposition enables the stochastic valuation function to be expressed as a weighted sum of stochastic valuation functions, with the weights corresponding to the probabilities of the mixture components. It is sufficient to discretize only the non-atomic mixture components using Algorithm~\ref{alg:disc}. The redefined sketch valuation function is then constructed by approximating each stochastic valuation function in the weighted sum with a corresponding sketch valuation function. This technique extends the approximation results in Theorem~\ref{DR}, Corollary~\ref{cor:cf}, and subsequent approximation results to accommodate item value distributions with arbitrary jump sizes. 

The above approach increases the representation size of each item by up to $1/\Delta$ elements (corresponding to the point mass distributions). The computational complexity of evaluating the sketch valuation function is upper bounded by 
$$
\left(\frac{1}{\Delta} + 1\right)^k s^k,
$$ 
where $s$ is an upper bound on the support size of each discretized distribution. A more detailed discussion is provided in the online companion.

\paragraph{Discussion of weakly homogeneous functions} Any homogeneous function $f$ of degree $1$ over $\Theta$ is clearly weakly homogeneous of degree $1$ and tolerance $\eta = 1$ over $\Theta$. For instance, $f(x) = \max\{x_1,\ldots, x_n\}$ and $f(x) = (\sum_{i=1}^n x_i^r)^{1/r}$ are homogeneous functions of degree $1$ over $\mathbb{R}$. Moreover, any convex function on a domain containing the origin and satisfying $f(0)\leq 0$ is weakly homogeneous of degree $1$ over $[0,1]$. 

Some concave functions are weakly homogeneous with a strictly positive degree. For example, $f(x) = (\sum_{i=1}^n x_i)^r$ with domain $\mathbb{R}_+^n$, for $r \in (0,1]$, is weakly homogeneous of degree $r$ over $\mathbb{R}_+$. 

A differentiable function $f$ is weakly homogeneous of degree $d$ over $[0,1]$ if and only if 
$$
x^\top \nabla f(x) \geq d\ f(x) \quad \text{for all } x\in \mathrm{dom}(f).
$$
For example, let $f(x) = g(\sum_{i=1}^n x_i)$, where $g$ is an increasing, differentiable, and concave function over $\mathbb{R}_+$. In this case, the above inequality is equivalent to requiring that the \emph{elasticity} of $g$, defined as 
$$
\eta(z) = \frac{z g'(z)}{g(z)},
$$ 
satisfies $\eta(z) \geq d$ for all $z\in \mathbb{R}_+$. The elasticity $\eta(z)$ of any such function $g$ is always less than or equal to $1$. Furthermore, $g$ has constant elasticity $r$ if and only if $g(z) = c z^r$ for some constant $c > 0$. Note that some concave functions may have zero minimum elasticity. For instance, $g(z) = 1-e^{-\lambda z}$, with $\lambda > 0$, has elasticity decreasing from $1$ at $z=0$ to $0$ as $z\rightarrow \infty$.

Many functions are weakly homogeneous over $[0,1]$ with a constant tolerance $\eta$. Any monotone subadditive function $f:\mathbb{R}^n\rightarrow \mathbb{R}_+$ is weakly homogeneous over $[0,1]$ with tolerance $\eta = 2$. This follows from the following lemma: 

\begin{lemma}\label{lem:fact2} 
If $f$ is a monotone and subadditive function on $\mathcal{X}\subseteq \mathbb{R}_+^n$, then for every $\lambda \in (0,1]$ and $x\in \mathcal{X}$, 
$$
f(x) \leq \left\lceil \frac{1}{\lambda} \right\rceil f(\lambda x).
$$
\end{lemma}

By this lemma, we have $f(\theta x)\geq (1/\lceil 1/\theta \rceil) f(x)$. Combined with the inequality 
$$
\frac{1}{\lceil 1/\theta \rceil} \geq \frac{1}{1/\theta +1} \geq \frac{\theta}{2},
$$
we obtain 
$$
\frac{1}{2}\, \theta f(x) \leq f(\theta x),
$$ 
i.e., $\eta = 2$. 

In addition, any function $f$ that is subadditive and convex over a domain containing the origin and satisfies $f(0)=0$ is weakly homogeneous over $[0,1]$ with tolerance $\eta = 1$. If $f(0)\leq 0$, the same conclusion still holds. This follows from the inequality 
$$
f(\theta x) \geq f(x) - f((1-\theta)x) \geq f(x) - (1-\theta) f(x) = \theta f(x),
$$
where the first inequality is by subadditivity and the second inequality by convexity. 

Finally, any concave function defined on a domain containing the origin and satisfying $f(0)\geq 0$ is weakly homogeneous with tolerance $\eta = 1$. This follows directly from the definition of concave functions.

\paragraph{The importance of the degree of homogeneity} In general, the dependence of the approximation factor on the degree parameter $d$ is unavoidable due to the truncation step that assigns values smaller than $a\tau_i$ to zero for item $i\in \Omega$. This can lead to a significant loss in approximation accuracy for functions with a small degree of weak homogeneity, particularly when the item value distributions place substantial mass near zero. We illustrate this phenomenon with the following example. 

Let $f(x) = x^r$ on the domain $\mathbb{R}_+$, for some $r \in (0,1]$. Let $X$ be a non-negative random variable with cumulative distribution function $P$, and assume that $P(\tau)=1-\epsilon$. Consider approximating $\E[X^r]$ by truncating the lower tail of the distribution and evaluating $\E[(X\ind{\{X\geq a\tau\}})^r]$. For any $a\in [0,1]$, we can write
$$
\E[(X\ind{X\geq a\tau})^r] = \rho\, \E[X^r],
$$
where
$$
\rho = \frac{\E[X^r \ind{X\geq a\tau}]}{\E[X^r]} = \frac{\int_{a\tau}^\infty x^r\, dP(x)}{\int_{0}^\infty x^r\, dP(x)}.
$$
Clearly, $\rho \in [0,1]$, and we desire $\rho$ to be as close to $1$ as possible to ensure a good approximation. However, there exist instances where $\rho$ can be arbitrarily close to zero, as shown below. 

Consider a distribution $P$ defined by $P(x) = x^d$ for $x\in [0,1]$, with $d > 0$. Then $\tau^d = 1-\epsilon$. Using basic calculus, we obtain
$$
\rho = 1 - \bigl(a (1-\epsilon)^{1/d}\bigr)^{r+d}.
$$
Suppose we choose $a = \epsilon^c$ for a fixed constant $c > 0$, and set $r = d = \epsilon$. Then,
$$
\rho = 1 - \epsilon^{2c\epsilon} (1-\epsilon)^{2\epsilon} \downarrow 0 \quad \text{as } \epsilon \downarrow 0.
$$
This demonstrates that, for this family of instances, $\E[(X\ind{X\geq a\tau})^r]$ can be an arbitrarily poor approximation of $\E[X^r]$ as $\epsilon \rightarrow 0$, particularly when the degree of homogeneity $r$ is small. Therefore, the degree parameter $d$ plays a crucial role in bounding the approximation loss, especially in settings with heavy lower tails in the distribution.  

\subsubsection{Extendable Concave Functions}

The weak homogeneity condition in Corollary~\ref{cor:cf} limits the approximation guarantees to valuation functions with strictly positive degrees of homogeneity. In this section, we show that similar guarantees can also be established for certain valuation functions that exhibit a zero degree of homogeneity.

\begin{definition} 
A monotone, subadditive, and concave function $f$ defined on $\mathbb{R}_+^n$ is said to admit an \emph{extension on $\mathbb{R}^n$} if there exists a function $f^*$ that is monotone, subadditive, and concave on $\mathbb{R}^n$ and satisfies $f^*(x) = f(x)$ for all $x\in \mathbb{R}_+^n$.
\label{def:econc}
\end{definition}

For example, consider the function $f(x) = g(\sum_{i=1}^n x_i)$ defined on $\mathbb{R}_+$ where $g(z) = 1-e^{-\lambda z}$ for $\lambda > 0$. This function has a weak homogeneity degree of zero and therefore falls outside the scope of Corollary~\ref{cor:cf}. However, $f$ admits an extension to $\mathbb{R}^n$. For instance, define $f^*(x) = g^*(\sum_{i=1}^n x_i)$, where
$$
g^*(z) = \begin{cases}
1-e^{-\lambda z}, & \text{if } z\geq 0,\\
\lambda z, & \text{otherwise}.
\end{cases}
$$

The following theorem establishes an approximation guarantee for functions $f$ that admit such extensions.

\begin{theorem} \label{thm:conc} 
Assume that $f$ is a monotone, subadditive, and concave function on $\mathbb{R}_+^n$ that admits an extension to $\mathbb{R}^n$, and let $\epsilon \in (\Delta,1)$. Then the discretization algorithm guarantees that for every set $S\subseteq \Omega$ with $|S|\leq k$, 
$$
\frac{1}{2}(1-\epsilon)^{k-1} (1-\Delta/\epsilon) v(S) \leq u(S)\leq 2\frac{1+ak/\epsilon - \Delta/\epsilon}{(1-\epsilon)^{k}(1-\Delta/\epsilon)^2}v(S).
$$
\end{theorem}

As a direct consequence, we obtain the following corollary.

\begin{corollary} 
Assume that $\Delta k < 1$. Under the same conditions as in Theorem~\ref{thm:conc}, and choosing $a=\epsilon^2$ and $\epsilon = c/k$ for some constant $c\in (\Delta k,1)$, we have that for every set $S\subseteq \Omega$ with $|S|\leq k$, 
$$
\frac{1}{2} e^{-\frac{c}{1-c}}(1-\Delta k/c) v(S)\leq u(S)\leq 2 e^{\frac{c}{1-c}} \frac{1+c}{(1-\Delta k/c)^2} v(S).
$$
\label{cor:conc}
\end{corollary}

Under the conditions of Corollary~\ref{cor:conc}, the sketch valuation function achieves an approximation factor $\alpha$ given by
$$
\alpha = 4\frac{1+c}{(1-\Delta k/c)^3}e^{2\frac{c}{1-c}},
$$
which can be made arbitrarily close to $4$ by selecting $c$ sufficiently small when $\Delta = 0$. 

Theorem~\ref{thm:conc} removes the requirement of the weak homogeneity condition for a subclass of concave functions, thereby extending the approximation guarantees to concave valuation functions that do not exhibit a strictly positive degree of homogeneity.  

However, not all concave functions on $\mathbb{R}_+^n$ admit an extension on $\mathbb{R}^n$. For example, consider $f(x) = g(\sum_{i=1}^n x_i)$, where $g$ has a vertical tangent at $z = 0$. If $g$ is differentiable at $z = 0$, then $\lim_{z\downarrow 0} g'(z) = \infty$. In this case, $f$ does not admit an extension. A concrete example is the power function $g(z) = z^r$ with $r \in (0,1)$, which is concave and strictly increasing on $\mathbb{R}_+$ but has infinite slope at the origin, violating the extension condition required by Theorem~\ref{thm:conc}.

\subsubsection{Coordinate-wise Conditions}

Some functions that fail to satisfy the weak homogeneity condition may still meet a more relaxed form, which we refer to as \emph{coordinate-wise weak homogeneity}.  

As an illustrative example, consider $x\in \mathbb{R}_+^2$ and the function $f(x) = x_1+x_1x_2$. For every $\theta \in [0,1]$, we have $f(\theta x) = \theta x_1 + \theta^2 x_1 x_2$, which does not satisfy $f(\theta x) \geq (1/\eta) \theta f(x)$ for some constant $\eta > 0$ for all $\theta\in [0,1]$ and all $x_1,x_2 \geq 0$. However, $f$ is coordinate-wise weakly homogeneous, as we show below.

\begin{definition} \label{def:whcor} 
A function $f$ is said to be \emph{coordinate-wise weakly homogeneous of degree $d$ with tolerance $\eta$ over a set $\Theta\subseteq \mathbb{R}$} if for every $i\in [n]$,
$$
(1/\eta)\ \theta f(x)\leq f\Bigl(\sum_{j\neq i} x_j e_j + \theta x_i e_i\Bigr)\leq \theta^d f(x),
$$
for all $x$ in the domain of $f$ and all $\theta \in \Theta$.
\end{definition}

In this section, we show that approximation guarantees can still be derived when the weak homogeneity condition holds coordinate-wise. 

\begin{theorem} \label{thm:coordinate} 
Assume that $f$ is a monotone subadditive or submodular function, and is coordinate-wise weakly homogeneous with degree $d$ and tolerance $\eta$ over $[0,1]$. Let $\epsilon \in (\Delta, 1)$. Then, the discretization algorithm guarantees that for every set $S\subseteq \Omega$ with $|S| \leq k$,
$$
\frac{1}{2}(1-\epsilon)^{k-1}(1-\Delta/\epsilon) v(S)\leq u(S)\leq 2\eta^k\frac{1+a^d k/\epsilon - \Delta /\epsilon}{(1-\epsilon)^{2k-1}(1-\Delta/\epsilon)^2}v(S).
$$
\end{theorem}

The proof is provided in the online companion.

From Theorem~\ref{thm:coordinate}, we derive the following corollary.

\begin{corollary} 
Assume that $\Delta k < 1$. Under the same conditions as in Theorem~\ref{thm:coordinate}, and assuming $\eta = 1$, let $a = \epsilon^{2/d}$ and $\epsilon = c/k$ for some constant $c\in (\Delta k,1)$. Then, for every set $S\subseteq \Omega$ with $|S|\leq k$,
$$
\frac{1}{2} e^{-\frac{c}{1-c}}(1-\Delta k/c) v(S)\leq u(S)\leq 2 \frac{1+c}{(1-\Delta k/c)^2} e^{\frac{c}{1-c}} v(S).
$$
\label{cor:coordinate2}
\end{corollary}

Finally, we note the following useful facts. First, any function $f$ that is subadditive and coordinate-wise convex on a domain containing $0$, with $f(0)=0$, is weakly homogeneous over $[0,1]$ with tolerance $\eta = 1$. Second, any function $f$ that is coordinate-wise concave on a domain containing $0$, with $f(0)\geq 0$, is also weakly homogeneous over $[0,1]$ with tolerance $\eta = 1$.

\subsubsection{Univariate Transformations}

For a given function $f$, we can sometimes derive stronger approximation guarantees by considering an associated function 
$$
f^*(x_1, \ldots, x_n) = f(\phi_1(x_1), \ldots, \phi_n(x_n)),
$$ 
where each $\phi_i: \mathbb{R}_+\rightarrow \mathbb{R}_+$ is continuous and strictly increasing. These univariate transformations correspond to a change of variables that affects only the input distributions, without altering the structure of the function class to which $f$ belongs. Such transformations are particularly useful when the original function does not directly satisfy the required conditions. 

We illustrate the benefits of this approach through two examples. 

\textbf{Example 1.} Let $f(x) = (\sum_{i=1}^n x_i)^r$ for $r\in (0,1)$. This function is weakly homogeneous over $[0,1]$ with degree $r$, so Corollary~\ref{cor:cf} yields a constant-factor approximation. However, the required support size for the discretized distributions is $O((1/r)k\log k)$, which grows as $r\to 0$. To eliminate the dependence on $r$, we apply the univariate transformation $\phi_i(z) = z^{1/r}$. The transformed function becomes
$$
f^*(x) = \Bigl(\sum_{i=1}^n x_i^{1/r}\Bigr)^r.
$$
The function $f^*$ is subadditive, submodular, convex, and weakly homogeneous over $[0,1]$ with degree $1$ and tolerance $1$. Therefore, by Corollary~\ref{cor:cf}, we obtain a constant-factor approximation with discretized distributions having support size $O(k\log k)$, independent of $r$.

\textbf{Example 2.} Consider $f(x) = 1-\prod_{i=1}^n (1-x_i)$ defined on $[0,1]^n$. This function is submodular and weakly homogeneous over $[0,1]$ with degree $d\leq 1/2$ and tolerance $1$. Applying Corollary~\ref{cor:cf} yields a constant-factor approximation, but the support size of the discretized distributions scales as $O((1/d)k\log k)$, which can be undesirable for small $d$. To remove this dependence, we apply the transformation $\phi_i(z)=1-e^{-z}$, leading to the transformed function
$$
f^*(x) = 1 - e^{-\sum_{i=1}^n x_i}.
$$
The function $f^*$ is concave and submodular with an extension on $\mathbb{R}^n$, so Corollary~\ref{cor:conc} applies. This gives a constant-factor approximation with discretized distributions of support size $O(k\log k)$, again avoiding dependence on the weak homogeneity degree $d$.

\section{Using Sketch Value Oracles in Discrete Optimization}
\label{sec:opt}

In this section, we demonstrate how the approximation guarantees of the sketch functions presented in the preceding sections make them suitable for approximately solving best set selection and welfare maximization problems. Before presenting these results, we formally define both optimization problems and recall well-known approximation algorithms that achieve constant-factor guarantees. We then describe how our sketch valuation functions can be combined with these algorithms.

The \emph{Best Set Selection} problem involves identifying a subset $S^*\subseteq \Omega$ of size $k$ that maximizes a given set function $u$, i.e.,
$$
S^* \in \arg\max_{S\subseteq \Omega: |S|=k} u(S).
$$
A set $S$ is said to be a $\rho$-approximate solution if it satisfies $u(S^*)\leq \rho u(S)$. If $v$ is an $\alpha$-sketch of $u$, and $S$ is a $\rho$-approximate solution for maximizing $v$, then $S$ also serves as a $\rho \alpha$-approximate solution for maximizing $u$. 

This result naturally extends to the more general \emph{Submodular Welfare Maximization} problem. Given $m$ agents with individual cardinality constraints $k_1, \ldots, k_m$, the objective is to find disjoint subsets $S_1, \ldots, S_m \subseteq \Omega$ with $|S_j|\leq k_j$ that maximize the total welfare:
$$
\sum_{j=1}^m u_j(S_j),
$$
where each $u_j$ is a monotone submodular set function.

It is well-known that the Greedy Algorithm achieves a $1-1/e$ approximation for Best Set Selection when the objective is monotone and submodular \cite{nemhauser}. The algorithm starts with the empty set and iteratively adds the item that provides the largest marginal gain in utility. For Submodular Welfare Maximization, a greedy allocation algorithm—commonly referred to as the Winner Determination Algorithm—provides a $2$-approximation guarantee \cite{lehmann}. 

We now state the following consequence of Corollary~\ref{cor:cf}.

\begin{corollary} 
Assume that $\Delta k < 1$. For any function class satisfying the conditions of Theorem~\ref{DR}, and by selecting $a = \epsilon^{2/d}$ and $\epsilon = c/k$ for some constant $c\in (\Delta k, 1)$, the greedy algorithms for Best Set Selection and Submodular Welfare Maximization achieve an approximation factor of
$$
4\eta\rho \frac{1+c}{(1-\Delta k/c)^3} e^{2\frac{c}{1-c}},
$$
where $\rho$ is a constant: $\rho = 1/(1-1/e)$ for Best Set Selection, and $\rho = 2$ for Submodular Welfare Maximization.
\label{cor:greedy}
\end{corollary}

When $\Delta = 0$, the approximation factor in Corollary~\ref{cor:greedy} can be made arbitrarily close to $4\eta \rho$ by choosing a sufficiently small constant $c$. As noted in Section~\ref{sec:approx}, for any given item value distributions, $\Delta$ can be made arbitrarily small by suitably extending the definition of the sketch value functions.

We now discuss the computational complexity. Consider a set valuation function defined as in (\ref{equ:sutil}), where each item's value is drawn from a discrete probability distribution with support of size at most $s$. Evaluating the expected utility $u(S)$ for a subset $S$ of cardinality $k$ requires computing over all combinations of value realizations, resulting in a computational complexity of $O(s^k)$ per evaluation. The Greedy Algorithm proceeds in $k$ iterations, and in iteration $t$, it evaluates $n-(t-1)$ candidate sets of size $t$. Hence, the total computational cost of the Greedy Algorithm is $O(n s^k)$. This complexity becomes linear in $n$ if both $s$ and $k$ are constants. More generally, the algorithm runs in polynomial time $O(n^{1+\gamma})$ for some constant $\gamma > 0$ if and only if $s^k = O(n^\gamma)$. For instance, this condition holds when $s = O(k \log k)$ and $k \leq \gamma \log n / \log \log n$.

\section{Numerical Results}
\label{sec:num}

We report numerical results that substantiate the accuracy of the sketch function, computed using discretized item-value distributions as outlined in Algorithm \ref{alg:disc}. The evaluation encompasses a range of set valuation functions, item-value distributions, set sizes, and parameter configurations of the discretization algorithm. Experiments were performed on both synthetic and real-world datasets. The code and datasets are available on GitHub: \url{https://github.com/Sketch-EXP/Sketch}.

\subsection{Experimental Settings}

\subsubsection{Goals}
We address two objectives: (1) the approximation of a stochastic valuation set function for arbitrary sets—referred to throughout as \emph{Function Approximation}; and (2) the approximate solution of the \emph{Best Set Selection Problem}, in which value oracle calls are evaluated using a sketch valuation function. The first objective forms the primary focus of our theoretical work, with results reported in Section~\ref{sec:algoapprox}. The second objective concerns the application of our sketch to discrete optimization, with corresponding theoretical guarantees presented in Section~\ref{sec:opt}. We next outline how these two objectives are evaluated in our experiments.
 
\paragraph{Function Approximation} For any given function $f$ and any subset $S$ of items, let $\hat{w}(S)$ denote the sample mean estimator of 
$w(S) = \mathbb{E}[f(X_S)]$, using $N$ samples of item values drawn from the 
underlying item–value distributions $F_i$ for $i \in S$. That is, 
$$
\hat{w}(S) = \frac{1}{N}\sum_{j=1}^N f(x_S^{(j)}),
$$
where $x_S^{(j)} = (x_{i,j},\, i \in S)$, and each $x_{i,j}$ is sampled from $F_i$.

We evaluate the accuracy of approximating a stochastic set valuation function $u$ by 
a sketch function $v$ by comparing the values of $\hat{u}(S)$ and $\hat{v}(S)$, 
where $S$ is sampled uniformly at random from all subsets of $k$ items. Here, 
$\hat{u}(S)$ is computed using samples from the given item–value distributions, 
while $\hat{v}(S)$ is computed using samples from the corresponding discretized 
distributions. Specifically, we consider the ratio $\hat{v}(S)\,/\,\hat{u}(S)$ 
for each sampled set $S$. The closer this ratio is to $1$, the better the approximation accuracy.
 
\paragraph{Best Set Selection Problem} 
We evaluate the performance of our sketch function when used as the value oracle 
for approximately solving the Best Set Selection optimization problem. We employ 
the classic Greedy Algorithm, which starts with an empty set of items and, at each 
step, selects a previously unselected item with the largest marginal gain relative 
to the items already chosen. This Greedy Algorithm is known to achieve an approximation ratio of $1 - 1/e$ for the class of monotone submodular set functions. 

Let $S_w$ denote the output of the Greedy Algorithm when using a value oracle that, 
for every input set $S$, returns $\hat{w}(S)$, where $\hat{w}(S)$ is the sample 
mean estimator of $w(S)$ based on $N$ samples of values for each item. We run the 
Greedy Algorithm twice: once with the value oracle $v$ and once with the value oracle $u$, obtaining two outputs, $S_v$ and $S_u$. We then compare the values $\hat{u}(S_v)$ and $\hat{u}(S_u)$, where for any given set $S$, $\hat{u}(S)$ denotes the sample mean estimate of $u(S)$ using $N$ samples of each item’s values from a held-out test set.

\subsubsection{Choice of Set Valuation Functions} 
For the choice of set valuation functions, we consider three examples: the Maximum Value Function $f(x)=\max\{x_1,\ldots, x_n\}$; the CES-$r$ Function $f(x)=(x_1^r+\cdots + x_n^r)^{1/r}$ for $r\geq 1$ (here we consider $r=2$); and the Square-Root Function $f(x) = \sqrt{x_1+\cdots + x_n}$. These functions were selected as they capture salient properties commonly arising in applications. The Maximum Value Function reflects the notion that the value of a set corresponds to its most relevant item, which is typical in information retrieval and recommender-system contexts. The CES function depends smoothly on the parameter $r$; when $r=1$, it reduces to an additive valuation, whereas as $r$ increases, greater relative weight is placed on the largest-value item. The CES function is widely used in economics, consumer theory, and production theory to model preferences and technologies with varying degrees of substitutability. The Square-Root Function captures the natural notion of diminishing returns via concave aggregation. All three functions are weakly homogeneous and thus fall within the scope of our theoretical results. As shown in Table~\ref{tab:fprop}, the Maximum Value and CES functions have degree $d=1$, and all three functions have tolerance $\eta=1$. The Square-Root Function has elasticity $1/2$ and consequently degree $1/2$.

\subsubsection{Choice of Item Value Distributions}
\paragraph{Synthetic Data} We consider two parametric families of item value distributions: the exponential and Pareto families. Considering these two families allows us to capture both light- and heavy-tailed distributions. In the former case, each item $i$ takes random values according to the exponential distribution with mean $1/\lambda_i$, i.e., $P_i(x) = 1 - e^{-\lambda_i x}$, for $x \geq 0$. In the latter case, each item $i$'s value distribution follows a Pareto distribution with scale parameter $\xi_i$ and shape parameter $\alpha_i$, i.e., $P_i(x) = 0$ for $x < \xi_i$ and $P_i(x) = 1 - (\xi_i/x)^{\alpha_i}$ for $x \geq \xi_i$. For Pareto distributions, we restrict attention to those with finite means, which corresponds to shape parameters $\alpha_i > 1$. Note that a Pareto distribution has infinite variance if and only if its shape parameter satisfies $\alpha_i \leq 2$.

\paragraph{Real-World Data} We consider the following datasets: YouTube, StackExchange, and the New York Times. In the YouTube dataset, items correspond to content publishers, with item values defined as the number of views attained by their published content. In the StackExchange dataset, items correspond to experts, with values reflecting the average number of upvotes received for their answers. In the New York Times dataset, items are treated as news sections, with performance measured by the number of comments per article. All datasets employed in this study are publicly available. For each item, we compute the empirical distribution of its performance based on the available data and generate random samples of item values from this empirical distribution for use in our experiments. Further information about the datasets and additional results are available in the online companion. 

\subsubsection{Comparison with the Test Score Benchmark} \label{sec:testscores} 
We compare the performance of our sketch function with that of the sketch function proposed in \cite{SVY21}, which, for any set $S$ of cardinality $k$, is defined as
$
v(S) = \sum_{j=1}^k a_{\pi_j(S),j}/j,
$
where $a_{i,j}$ are the test scores defined by
$
a_{i,j} := \E[f((X_i^{(1)}, \ldots, X_i^{(j)}))],
$
with $X_i^{(1)}, \ldots, X_i^{(j)}$ being independent and identically distributed random variables with distribution $P_i$, and
$
\pi_j(S) = \arg\max \{a_{i,j} : i \in S \setminus \{\pi_1(S),\ldots, \pi_{j-1}(S)\}\},
$
for $j=1,\ldots,k$. As shown in \cite{SVY21}, this sketch function satisfies the following guarantee for every set of cardinality $k$ in the class of valuation functions satisfying the extended diminishing returns condition: $(1/(2(\log(k)+1))) u(S) \leq v(S) \leq 6\, u(S)$. In our experiments, the test scores are estimated using sample mean estimators with 
$N$ samples of each item’s values.

\begin{figure}[t!]
\centering
\begin{tabular}{c}
{\footnotesize (a) Exponential distribution case}\\
\includegraphics[width=0.8\textwidth]{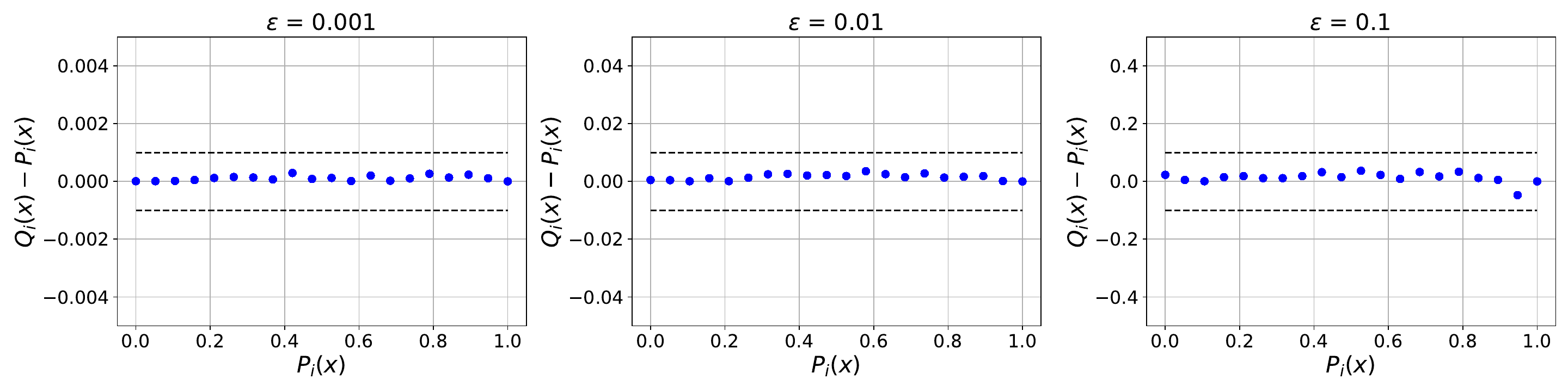}\\
{\footnotesize (b) Pareto distribution case}\\
\includegraphics[width=0.8\textwidth]{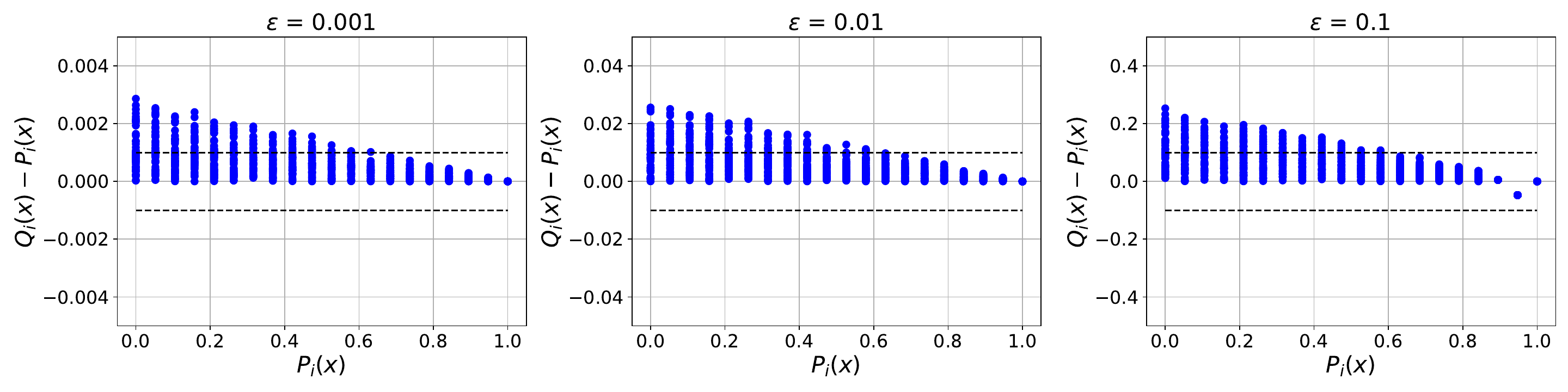}
\end{tabular}

\caption{The difference between distributions $Q_i$ and $P_i$ evaluated at different quantiles of distribution $P_i$. 
The horizontal dash lines indicate $\pm\epsilon$ levels.}
\label{fig:CDF_diff}
\end{figure}
\subsection{Results}

In our experiments, we employ the discretization Algorithm~\ref{alg:disc} with parameters $a$ and $\epsilon$ set as $a = \epsilon^2$ and $\epsilon = c/k$, where $c$ is a positive constant whose value is specified in the context of each experiment. As noted in Section~\ref{sec:algoapprox}, for this choice of parameters, the support size of the discretized distributions is $O(k \log k)$.

\subsubsection{Synthetic Data} \label{synthetic}

We consider a ground set of $n = 50$ items. In the case of exponential distributions, the mean values uniformly span the interval $[0.01,1]$. For the Pareto distributions, the scale parameters are fixed at $1.5$, while the shape parameters uniformly span the interval $[1.1,3]$. This range of shape parameter values covers distributions with both finite and infinite variance. For each item, we generate $N=10{,}000$ training samples of random performance values, which are used to estimate the set valuation function.

Before presenting the main results of our numerical evaluation, we first demonstrate how well the discretized distribution $Q_i$, computed by Algorithm~\ref{alg:disc}, approximates the input distribution $P_i$, for different values of the parameter \(\epsilon\). Figure~\ref{fig:CDF_diff} shows the gap between the discretized distribution \(Q_i\) and \(P_i\), evaluated at the \(p\)-quantiles of \(P_i\), with \(p\) uniformly spanning the unit interval. These results confirm that \(Q_i(x) - P_i(x) \geq -\epsilon\) for all \(x\) and that the gap between the two cumulative distribution functions shrinks towards zero as $\epsilon$ decreases. For the case of Pareto distribution, the gap tends to be larger for smaller quantile values.


\begin{figure}[t]
\centering

\subcaptionbox{Exponential distribution case\label{fig:boxexp}}{
\includegraphics[width=0.3\linewidth]{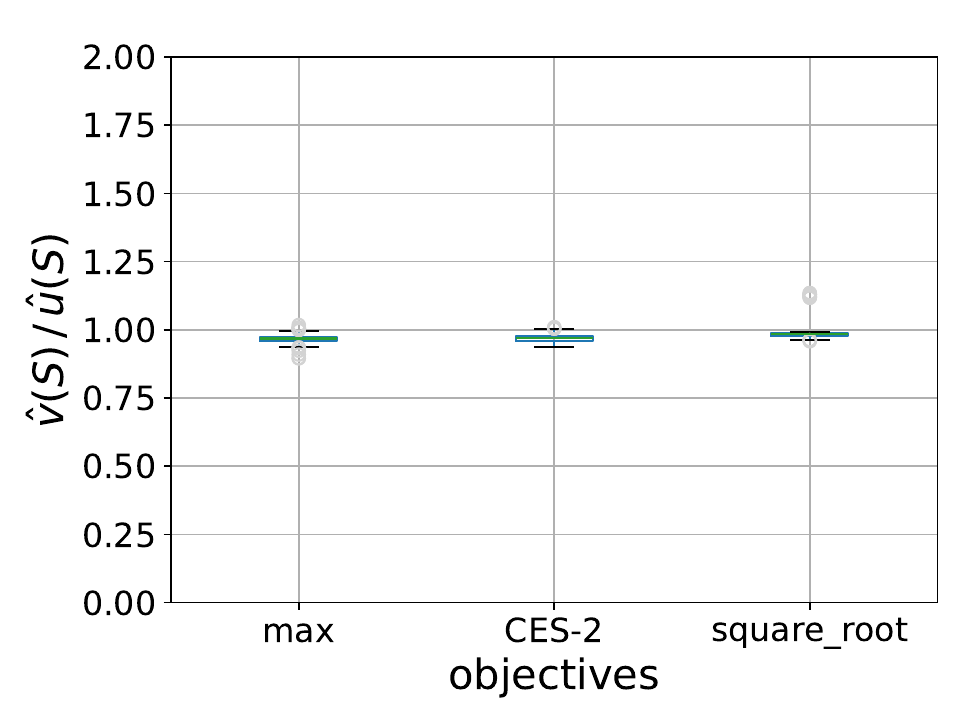}
}
\hspace{2cm}
\subcaptionbox{Pareto distribution case\label{fig:boxpar}}{
\includegraphics[width=0.3\linewidth]{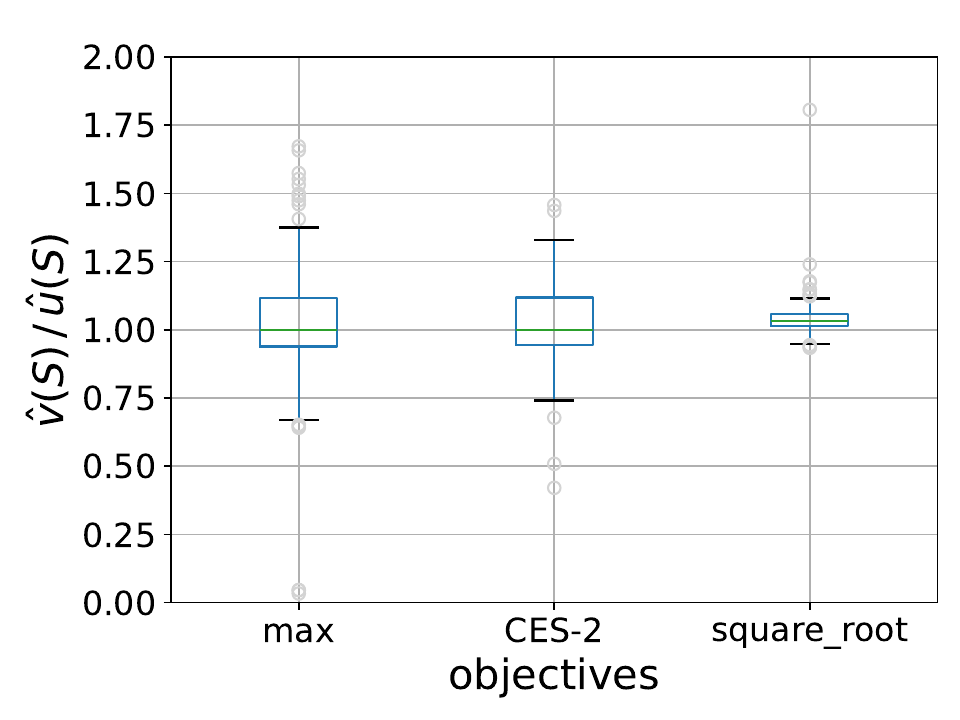}
}

\caption{Function approximation accuracy for various objective functions and item value distributions. 
The median approximation ratios are close to 1.}
\label{fig:box}

\end{figure}
\paragraph{Function Approximation} Recall that we consider the ratio $\hat{v}(S)\, /\, \hat{u}(S)$, where $\hat{v}(S)$ is the sample mean estimator of $v(S)$ with samples drawn from the discretized distributions produced by Algorithm~\ref{alg:disc}, and $\hat{u}(S)$ is the sample mean estimator of $u(S)$ with samples drawn from the original distributions. This ratio is computed for a sample of 10 sets, uniformly drawn from all subsets of $k$ items of the ground set. For the parameter $k$, we consider all values in the range from 1 to 20.

In Figure~\ref{fig:box}, we present results obtained for the parameter $\epsilon$ of Algorithm~\ref{alg:disc}, set as $\epsilon = c/k$ with $c = 0.1$. We observe that the median value of the ratio $v(S)/u(S)$ tends to be close to $1$ for all item value distributions and functions considered. Specifically, the ratio is tightly concentrated around $1$ for the exponential distribution, whereas it exhibits a wider range of values for the Pareto distribution, for both the maximum value and CES-2 functions.

Intuitively, with $\epsilon$ set as $\epsilon = c/k$, the approximation accuracy is expected to deteriorate for sufficiently large $c$. Our theoretical results suggest that this deterioration may occur when $\epsilon$ exceeds a value of the order $1/k$. To verify this experimentally, we consider $c$ uniformly distributed over the range from 0.1 to 10. The results shown in Figure~\ref{fig:eps} clearly support both our intuition and the theoretical predictions. Notably, the approximation accuracy tends to deteriorate as the value of $c$ increases. The approximation remains nearly exact when $\epsilon$ is set to $1/k$.

\begin{figure}[t]
\centering
\begin{tabular}{c}
{\footnotesize (a) Exponential distribution case}\\
\includegraphics[width=0.8\textwidth]{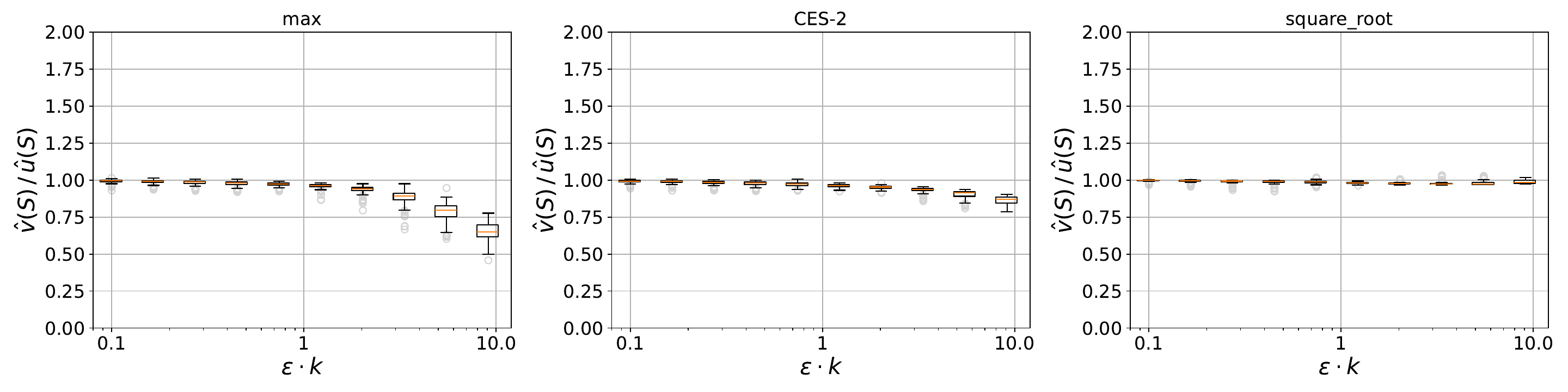}\\
{\footnotesize (b) Pareto distribution case}\\
\includegraphics[width=0.8\textwidth]{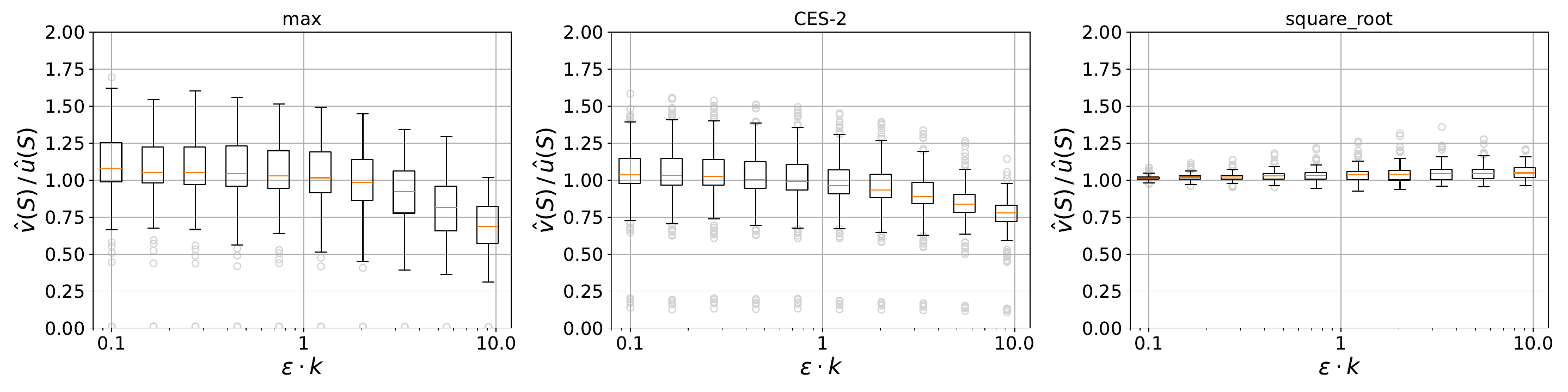}
\end{tabular}

\caption{Function approximation accuracy with varied parameter $\epsilon$. 
As $\epsilon$ increases, the support size decreases, at the expense of lower approximation accuracy. 
The approximation accuracy deteriorates around $\epsilon = 1/k$, regardless of the set size $k$.}
\label{fig:eps}
\end{figure}
For comparison with the test score benchmark, described in Section~\ref{sec:testscores}, we consider our algorithm with $\epsilon = c/k$, where $c = 0.99$. The results in Figure~\ref{fig:testscore-funcapprox} show that our method achieves an approximation ratio close to $1$, whereas the benchmark method tends to significantly overestimate the function values. 

\begin{figure}[t]
\centering

\subcaptionbox{Exponential distribution case\label{fig:tsexp}}{
\includegraphics[width=0.3\linewidth]{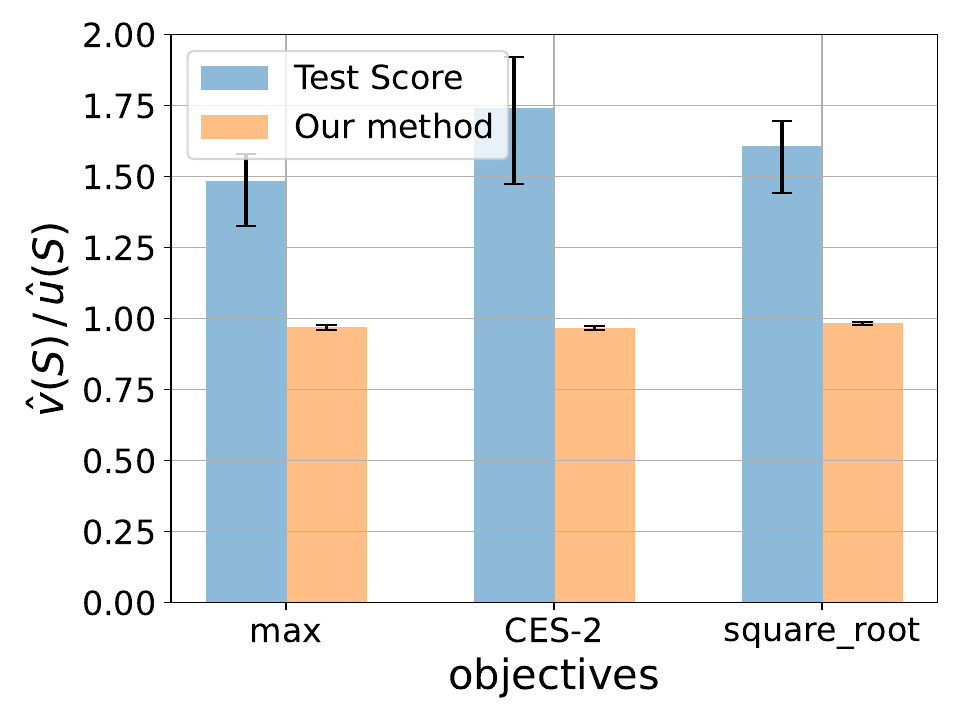}
}
\hspace{2cm}
\subcaptionbox{Pareto distribution case\label{fig:tspar}}{
\includegraphics[width=0.3\linewidth]{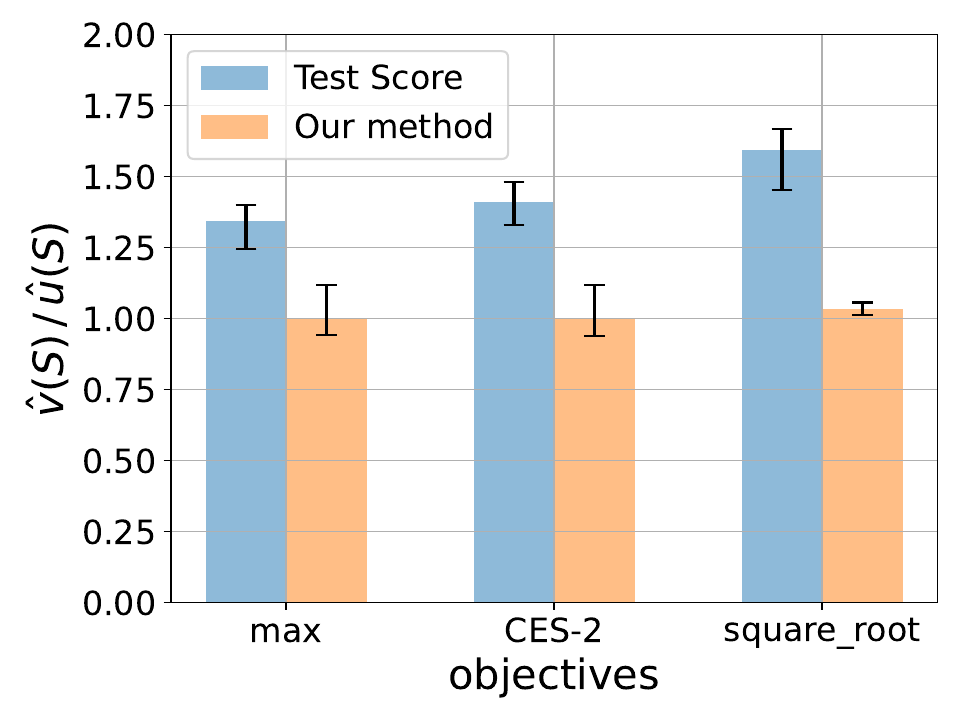}
}

\caption{Function approximation accuracy comparison with the test score benchmark. 
The bars show median values, while the error bars indicate the first and third quartiles.}
\label{fig:testscore-funcapprox}

\end{figure}

\paragraph{Best Set Selection Problem} 
We evaluate the performance of the Greedy Algorithm in approximately solving the best set selection problem, using our sketch function as the value oracle, and compare it with the Greedy Algorithm when the original value function is used as the oracle. For the cardinality constraint $k$, we consider all values in the range from $1$ to $10$. The assumed item–value distributions follow a stochastic order (first-order stochastic dominance), such that the optimal set of items corresponds to those with the largest mean values (for the Pareto case, this corresponds to the distributions with the smallest shape parameters). We examine the solution value as a function of the parameter $c$, where $c = \epsilon \cdot k$. 
The set valuations are computed using sample mean estimators with $N = 10{,}000$ samples of values for each item.

The results in Figure~\ref{fig:opt_eps} show that, for any sufficiently small value 
of $c$, the median approximation ratios are close to $1$. Moreover, the sketch function generally yields significantly superior performance, particularly for the 
Pareto distribution and for the maximum value or CES-2 valuation functions.

\begin{figure}[t]
\centering
\begin{tabular}{c}
{\footnotesize (a) Exponential distribution case}\\
\includegraphics[width=0.8\textwidth]{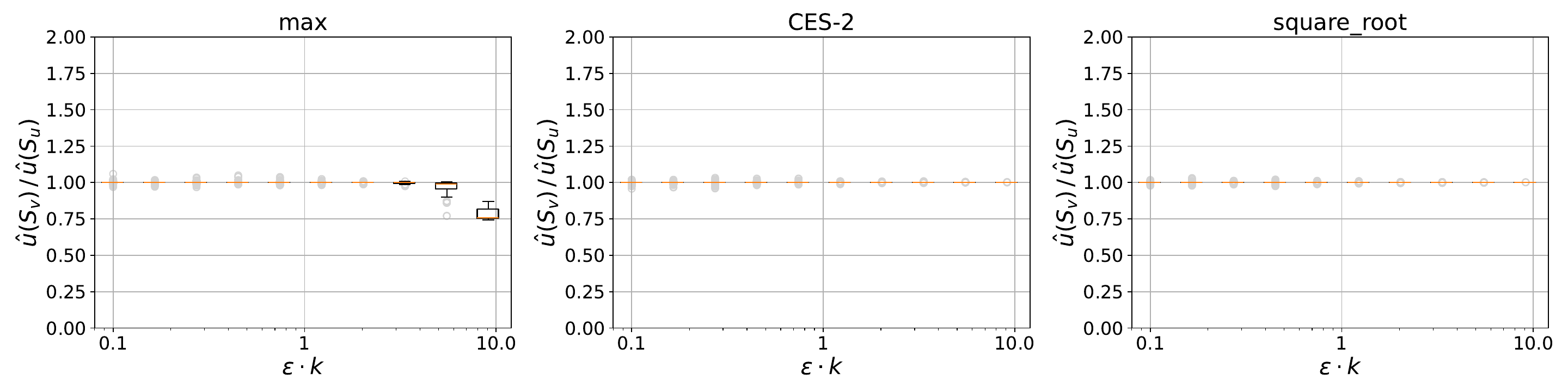}\\
{\footnotesize (b) Pareto distribution case}\\
\includegraphics[width=0.8\textwidth]{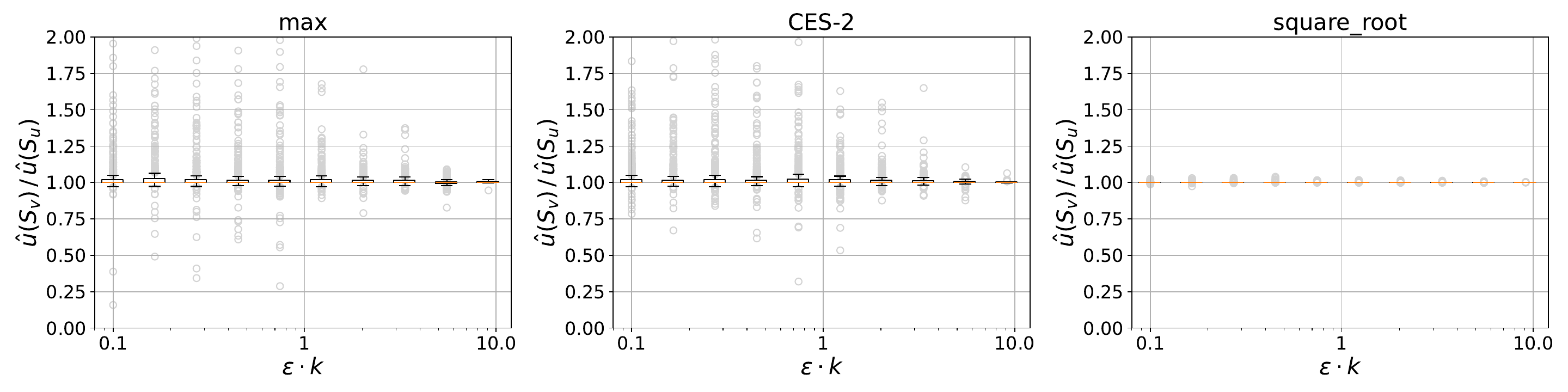}
\end{tabular}

\caption{Best set selection performance for varied parameter $\epsilon$. 
For the Pareto case, note that the lower quartile values are largely above $1$ for both the maximum value 
and CES-2 valuation functions.}
\label{fig:opt_eps}
\end{figure}

\begin{figure}[t]
\centering

\subcaptionbox{Exponential distribution case\label{fig:testfunexp}}{
\includegraphics[width=0.3\textwidth]{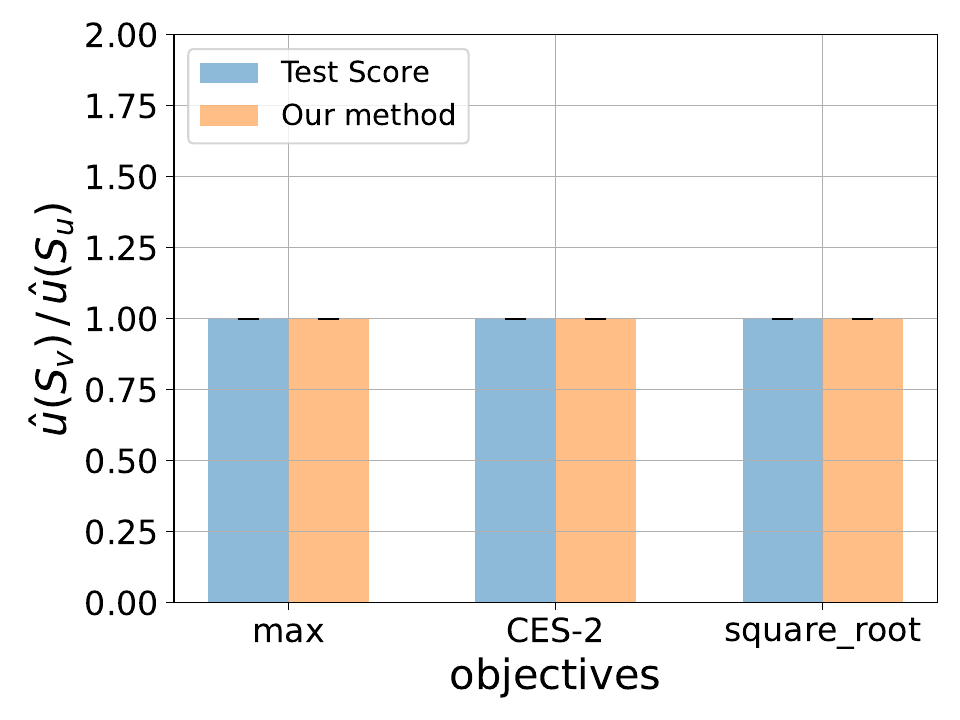}
}
\hspace{1cm}
\subcaptionbox{Pareto distribution case\label{fig:testfuncpar}}{
\includegraphics[width=0.3\textwidth]{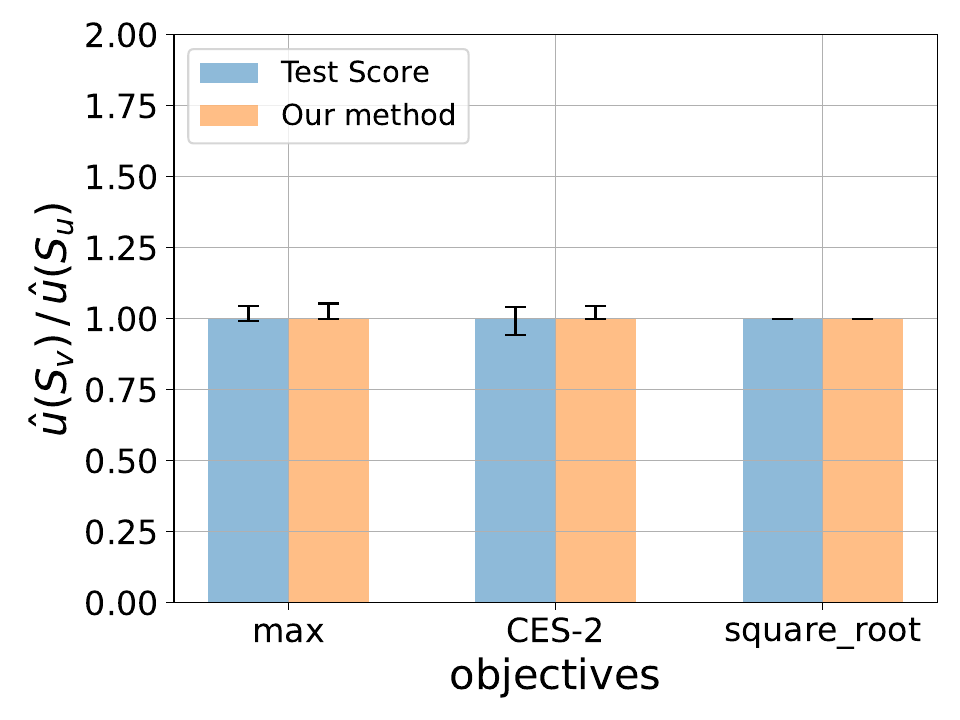}
}

\caption{Best set selection performance comparison with the test score benchmark. 
The bars represent median values, while the error bars indicate the first and third quartiles.}
\label{fig:testscore-funcapprox-opt}

\end{figure}

For comparison with the test score benchmark, we apply our sketch with $c = 0.99$. In Figure~\ref{fig:testscore-funcapprox-opt}, we observe that both methods identify a solution whose value is nearly the same as that found by the Greedy Algorithm using exact value oracles (up to statistical estimation noise). The fact that the test score sketch function is accurate for this problem instance can be explained by the item value distributions obeying first-order stochastic dominance, which implies that the optimal set of items corresponds to that which maximizes the test score valuation function.

\subsubsection{Real-world Data} \label{real}

\begin{figure}[t!]
\centering
\begin{tabular}{c}
{\footnotesize (a) StackEx case}\\
\includegraphics[width=0.8\textwidth]{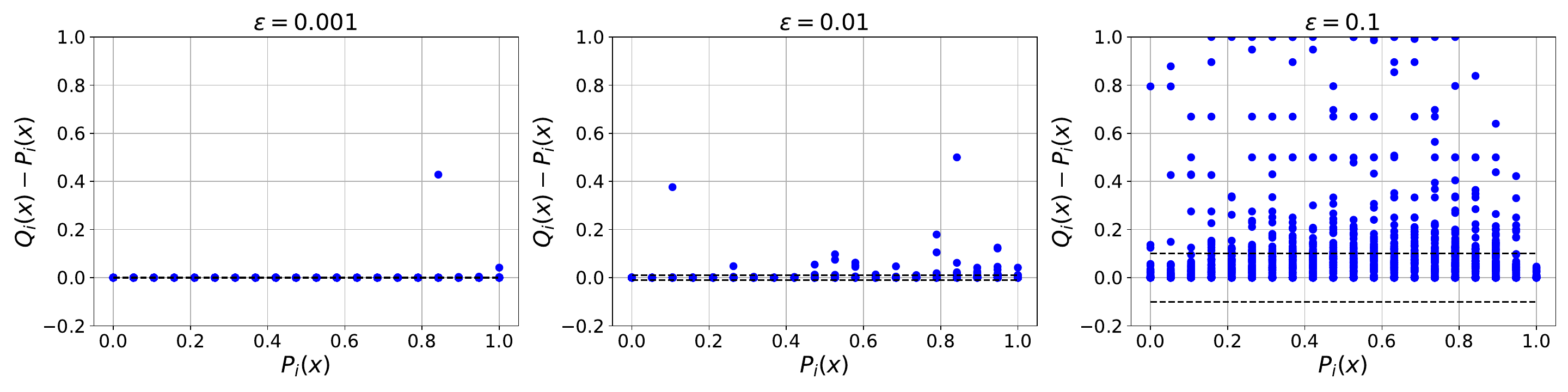}\\
{\footnotesize (b) YouTube case}\\
\includegraphics[width=0.8\textwidth]{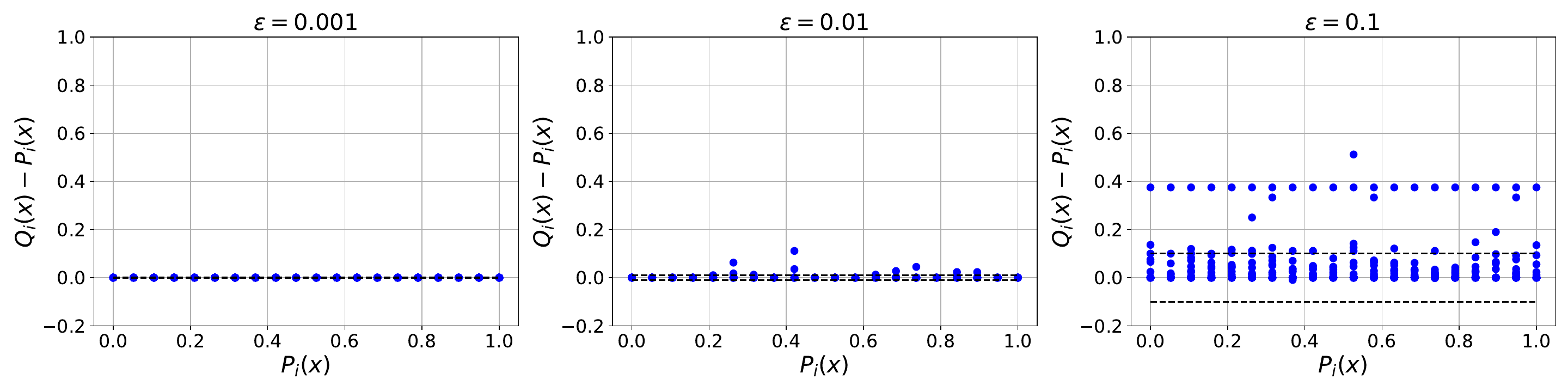}\\
{\footnotesize (c) Times case}\\
\includegraphics[width=0.8\textwidth]{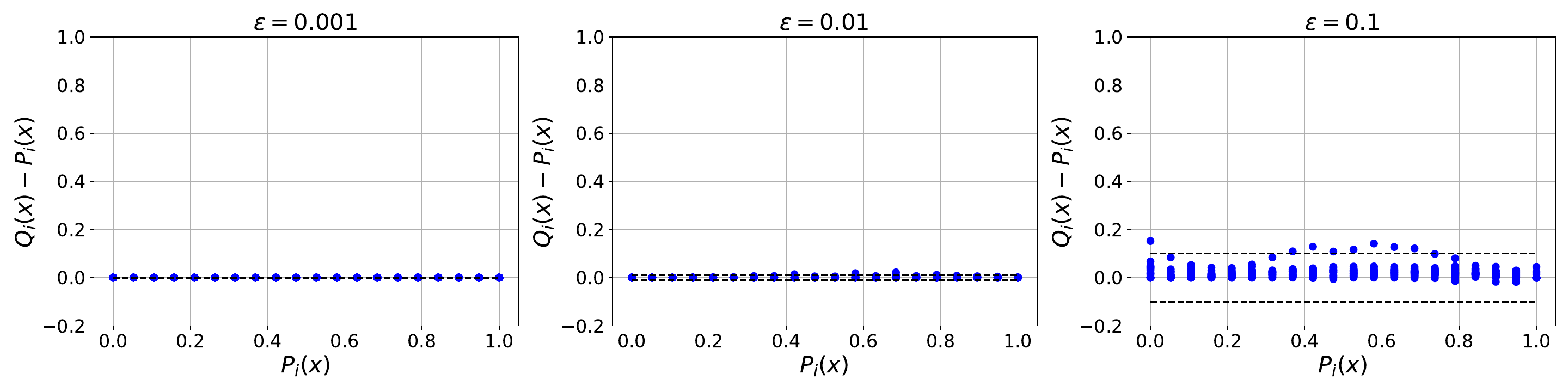}
\end{tabular}

\caption{The difference between distributions $Q_i$ and $P_i$ evaluated at different quantiles of distribution $P_i$. 
The horizontal dash lines indicate $\pm\epsilon$ levels.}
\label{fig:CDF_diff_realworld}
\end{figure}

We first report results on the gap between the input and output distributions of Algorithm~\ref{alg:disc} for three values of the parameter~$\epsilon$. Figure~\ref{fig:CDF_diff_realworld} confirms that the discretized distribution~$Q_i$ first-order stochastically dominates~$P_i$, and that the gap decreases as~$\epsilon$ becomes smaller. In some cases, the gap is larger than that observed for the Exponential and Pareto distributions shown in Figure~\ref{fig:CDF_diff}, which is expected since the real-world input distributions are discrete, leading to localized discrepancies near bin boundaries.

For the sketching approximation evaluation, we run Algorithm~\ref{alg:disc} with $a=\epsilon^2$ and $\epsilon=1/10$, considering set sizes of~$5$ and~$10$. Set valuations are estimated using $N=10{,}000$ training samples per item. Figure~\ref{fig:data} reports the results for the three valuation functions under consideration. Overall, the proposed sketching function achieves high approximation accuracy across valuation functions and set sizes.


\begin{figure}[t]
\centering

\subcaptionbox{Set size $k=5$\label{fig:data5}}{
\includegraphics[width=0.3\textwidth]{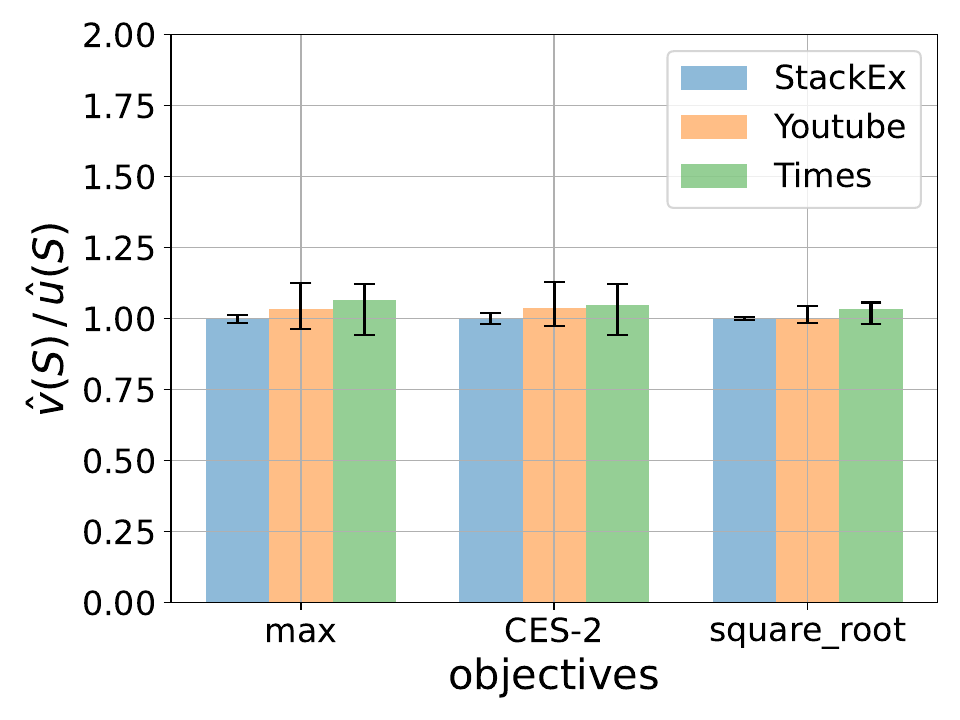}
}
\hspace{1cm}
\subcaptionbox{Set size $k=10$\label{fig:data10}}{
\includegraphics[width=0.3\textwidth]{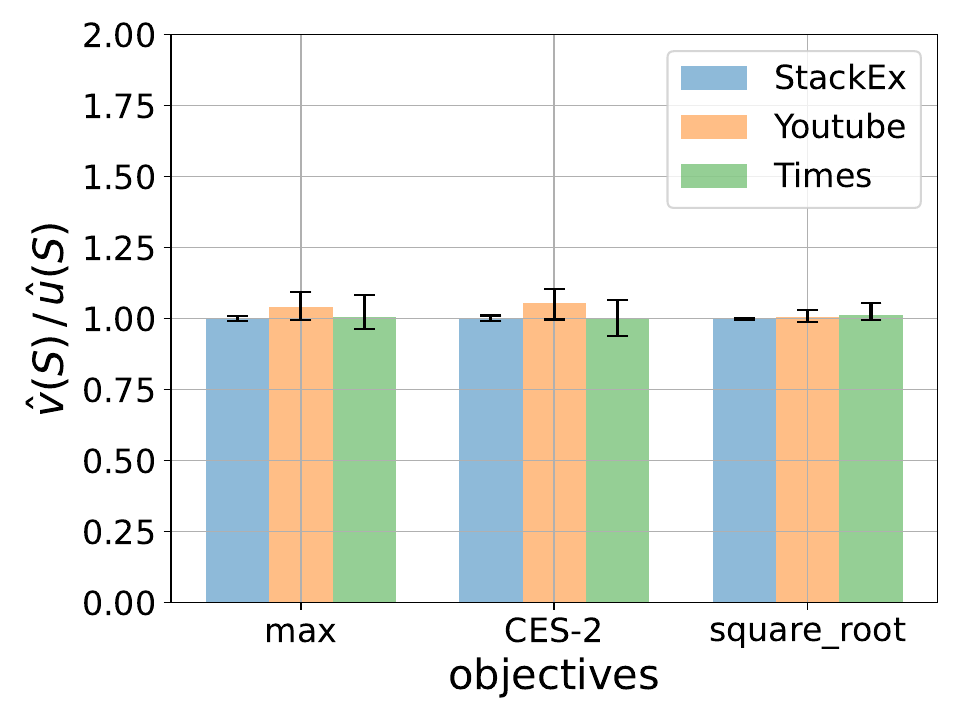}
}

\caption{Function approximation results for distributions derived from three real-world datasets. 
The bars denote median values, while the error bars correspond to the first and third quartiles.}
\label{fig:data}

\end{figure}

\section{Conclusion}
\label{sec:conc}

In this paper, we addressed the problem of finding an effective sketch of a stochastic set valuation function, defined as the expectation of a valuation function over independent random item values. We introduced an efficient algorithm to compute the sketch valuation function by expressing it as the expectation of the valuation function with respect to discretized item value distributions. Our analysis shows that, for a broad class of monotone subadditive or submodular valuation functions satisfying certain conditions, the algorithm achieves a constant-factor approximation for any value query involving sets of items of size at most $k$. Notably, our approach employs discretized distributions with support sizes of $O(k\log(k))$.

Our work provides the first positive results on function approximation for class functions that can accommodate a wide range of valuation functions studied in the existing literature. The results are also relevant for applications in best set selection and welfare maximization problems. As part of future work, it would be interesting to explore alternative systematic discretization strategies and examine the trade-off between approximation accuracy and the complexity of the sketch.


\section*{Acknowledgment}
This work was conducted as part of a thesis at the London School of Economics.


\bibliography{ref} 





%
%
%

\newpage

\appendix

\section{Further Properties of Valuation Functions}

In this appendix, we provide additional properties and examples of valuation functions discussed in the main text.

\subsection{Not All Submodular Functions Are Subadditive}

Consider, for example, the success-probability value function:
\begin{equation*}
f(x) = 1 - \prod_{i=1}^n \big(1 - p(x_i)\big),
\end{equation*}
where $p: \mathbb{R} \rightarrow [0,1]$ is an increasing function. The submodularity of this function can be verified by checking its satisfaction of the weak DR property. For any $x, y \in \mathbb{R}^n$ such that $x \leq y$, the increasing nature of $p$ implies that $f(x) \leq f(y)$. 

To confirm the weak DR property, consider adding $z$ to the $j$-th basis direction of both $x$ and $y$, such that $x \leq y$ and $x_j = y_j$. The weak-DR condition becomes: 
$$
\prod_{i \neq j} (1 - p(x_i)) \, (1 - p(x_j + z)) 
\geq  \prod_{i \neq j} (1 - p(y_i)) \, (1 - p(y_j + z)).
$$
This condition clearly holds since $x_j = y_j$ and $f(x) \leq f(y)$. 

However, for certain choices of the function $p$, $f$ may not be subadditive. For instance, when $n = 1$, $f$ is subadditive if and only if $p$ is subadditive.

\subsection{Coordinate-Wise Concave Functions}

There exist functions that are coordinate-wise concave but are not concave in the classical sense. An example is the maximum value function 
\begin{equation*}
f(x) = \max\{x_1, \ldots, x_n\}, \quad n > 1,
\end{equation*} 
which is, in fact, coordinate-wise \emph{convex} in the standard sense of convex functions. 

On the other hand, a function that is concave in the classical sense (and therefore coordinate-wise concave) is
\begin{equation*}
f(x) = g\Big(\sum_{i=1}^n x_i\Big),
\end{equation*}
where $g$ is concave.

\subsection{Extended Diminishing Returns Property}

A function $f$ is said to satisfy the \emph{extended diminishing returns property} \cite{SVY21} if, for any $i \in [n]$ and $v \geq 0$ that has a non-empty preimage under $f$, there exists $y \in \mathbb{R}_+^n$ with $y_i = 0$ such that: 

(a) $f(y) = v$, and  
(b) $f(x + z e_i) - f(x) \geq f(y + z e_i) - f(y)$ for any $z \in \mathbb{R}$ and any $x$ such that $f(x) \leq f(y) = v$ and $x_i = 0$.  

A simpler, stronger version is that $f$ satisfies 
\begin{equation*}
f(x + z e_i) - f(x) \geq f(y + z e_i) - f(y)
\end{equation*}
for every $z \in \mathbb{R}$ and all $x, y$ such that $f(x) \leq f(y)$ and $x_i = y_i = 0$.

There exist functions that satisfy the extended diminishing returns property but are not DR-submodular. For example, consider 
\begin{equation*}
f(x) = \Big(\sum_{i=1}^n x_i^r\Big)^{1/r}, \quad r > 1.
\end{equation*}
This function satisfies the extended diminishing returns property, as shown in \cite{SVY21}. However, $f$ is not DR-submodular: it is twice-differentiable and convex, hence coordinate-wise convex in the standard sense. In contrast, twice-differentiable DR-submodular functions are coordinate-wise concave in the classical sense.

\subsection{Success Probability Function}
\label{sec:success}

We consider the function 
$$
f(x) = 1 - \prod_{i=1}^n (1 - x_i)
$$ 
on $[0,1]^n$. 

\paragraph{Submodularity} This function is clearly submodular, as it is twice differentiable and 
$$
\frac{\partial^2 f(x)}{\partial x_i \partial x_j} = 
\begin{cases}
- \prod_{l \in [n] \setminus \{i,j\}} (1 - x_l) \leq 0, & i \neq j,\\
0, & i = j.
\end{cases}
$$

We show some properties of $f$ by induction over the sequence of functions $f_1, \ldots, f_n$, where 
$$
f_j(x) = 1 - \prod_{i=1}^j (1 - x_i), \quad 1 \leq j \leq n.
$$
Note that, for $1 \leq j < n$, 
$$
f_{j+1}(x) = x_{j+1} + f_j(x) - x_{j+1} f_j(x).
$$

\paragraph{Subadditivity} We show that $f$ is subadditive by induction. Let $x, y \in [0,1]^n$ such that $x+y \in [0,1]^n$. For the base case $j = 1$, $f_1(x) = x_1$ is clearly subadditive. For the induction step, assume $f_j$ is subadditive for some $1 \leq j < n$. Then:
\begin{align*}
f_{j+1}(x+y) 
&= x_{j+1} + y_{j+1} + f_j(x+y) - (x_{j+1}+y_{j+1}) f_j(x+y)\\
&= x_{j+1} + y_{j+1} + (1 - x_{j+1} - y_{j+1}) f_j(x+y)\\
&\leq x_{j+1} + y_{j+1} + (1 - x_{j+1} - y_{j+1}) (f_j(x) + f_j(y))\\
&= f_{j+1}(x) + f_{j+1}(y) - x_{j+1} f_j(y) - y_{j+1} f_j(x)\\
&\leq f_{j+1}(x) + f_{j+1}(y),
\end{align*}
which shows that $f_{j+1}$ is subadditive.

\paragraph{Weak Homogeneity} We next show that $f$ is weakly homogeneous over $[0,1]$ with tolerance $\eta = 1$ by induction. For the base case $j=1$, $f_1(x) = x_1$, so clearly $f_1(\theta x) \geq \theta f_1(x)$. For the induction step, assume $f_j(\theta x) \geq \theta f_j(x)$ for some $1 \leq j < n$. Then:
\begin{align*}
f_{j+1}(\theta x) 
&= \theta x_{j+1} + f_j(\theta x) - \theta x_{j+1} f_j(\theta x)\\
&= \theta x_{j+1} + (1 - \theta x_{j+1}) f_j(\theta x)\\
&\geq \theta x_{j+1} + (1 - \theta x_{j+1}) \theta f_j(x)\\
&= \theta x_{j+1} + \theta f_j(x) - \theta^2 x_{j+1} f_j(x)\\
&\geq \theta x_{j+1} + \theta f_j(x) - \theta x_{j+1} f_j(x)\\
&= \theta f_{j+1}(x),
\end{align*}
which shows that $f_{j+1}$ is weakly homogeneous over $[0,1]$ with tolerance $\eta = 1$.

We also show that $f$ is weakly homogeneous over $[0,1]$ with degree $d \leq 1/2$. Consider the case $n=2$, where 
$$
f(x) = x_1 + x_2 - x_1 x_2.
$$
The condition $f(\theta x) \leq \theta^d f(x)$ can be written as
$$
(1 - \theta^{2-d}) x_1 x_2 \leq (1 - \theta^{1-d}) (x_1 + x_2)
$$
for all $x_1, x_2 \in [0,1]$. This clearly holds when $x_1 = 0$ or $x_2 = 0$. Hence, it is equivalent to
$$
1 - \theta^{2-d} \leq (1 - \theta^{1-d}) \left( \frac{1}{x_1} + \frac{1}{x_2} \right)
$$
for all $x_1, x_2 \in (0,1]$, which reduces to 
$$
1 - \theta^{2-d} \leq 2 (1 - \theta^{1-d}).
$$
This can be rewritten as
\begin{equation}
\theta^{1-d} (2 - \theta) \leq 1.
\label{equ:thetad}
\end{equation}
The left-hand side is increasing in $d$ and achieves its maximum at $\theta^* = 1/(2(1-d))$. Hence, equality in (\ref{equ:thetad}) is achieved when $d = 1/2$.

\section{Proofs of Lemmas and Theorems}

\subsection{Proof of Lemma~\ref{lem:fact1}}
\label{app:fact1}

For any $x, y \in \mathcal{X}$, we have
\begin{align*}
f(x+y) - f(x) 
=& f\left(x + \sum_{i=1}^n y_i e_i\right) - f\left(x + \sum_{i=2}^n y_i e_i\right)\\ 
&+ f\left(x + \sum_{i=2}^n y_i e_i\right) - f\left(x + \sum_{i=3}^n y_i e_i\right)\\
&\vdots \\ 
&+ f(x + y_n e_n) - f(x) \\
&\leq f\left(\sum_{i=1}^n y_i e_i\right) - f\left(\sum_{i=2}^n y_i e_i\right)\\
&+ f\left(\sum_{i=2}^n y_i e_i\right) - f\left(\sum_{i=3}^n y_i e_i\right)\\
&\vdots \\
&+ f(y_n e_n) - f(0)\\
&= f(y) - f(0),
\end{align*}
where the inequalities follow from the DR-submodular property. Combining with $f(0) \geq 0$, we conclude that $f(x+y) - f(x) \leq f(y)$, which shows that $f$ is subadditive on $\mathcal{X}$.

\subsection{Proof of Lemma~\ref{lem:sub}}
\label{app:sub}

If $f$ is a monotone submodular function, then by Lemma~\ref{lem:AN}, $u$ is a monotone submodular set function; hence, $u$ is also monotone subadditive. 

Now, consider the case when $f$ is a monotone subadditive function. For any $S, T \subseteq \Omega$, we have
\begin{align*}
u(S) + u(T) 
&= \E[f(X_S)] + \E[f(X_T)] \\
&= \E\left[f\left(\sum_{i \in S} X_i e_i\right)\right] + \E\left[f\left(\sum_{i \in T} X_i e_i\right)\right].
\end{align*}

By the monotonicity and subadditivity of $f$, for every $x$ in the domain of $f$,
\begin{align*}
f\left(\sum_{i \in S} x_i e_i\right) + f\left(\sum_{i \in T} x_i e_i\right) 
&\geq f\left(\sum_{i \in S} x_i e_i + \sum_{i \in T} x_i e_i\right) \\
&\geq f\left(\sum_{i \in S \cup T} x_i e_i\right).
\end{align*}
Thus, it follows that
$$
u(S) + u(T) \geq \E[f(X_{S \cup T})] = u(S \cup T).
$$

\subsection{Proof of Lemma~\ref{lem:fact2}}
\label{app:fact2}

By monotonicity, $f(x) \leq f(\lceil 1/\lambda \rceil \lambda x)$, and by subadditivity, $f(\lceil 1/\lambda \rceil \lambda x) \leq \lceil 1/\lambda \rceil f(\lambda x)$. Combining these, we conclude that
$$
f(x) \leq \lceil 1/\lambda \rceil f(\lambda x).
$$

\subsection{Comparison of the Input and Output Distributions for Algorithm~\ref{alg:disc}} 
\label{sec:prop}

We present some comparison properties for distributions $P$ and $Q$, where $Q$ is the output of Algorithm~\ref{alg:disc} for the input distribution $P$ in the following lemma.

\begin{lemma} \label{lem:comp} 
For two distributions $P$ and $Q$, where $Q$ is the output of Algorithm~\ref{alg:disc} for the input distribution $P$, the following properties hold:
\begin{enumerate}
\item[(i)] $Q(x) \geq P(x) - \epsilon$ for all $x$,
\item[(ii)] $Q(x) \geq P(x)$ for all $x \leq \tau$, and
\item[(iii)] $Q(x) \leq P(x) + \epsilon$ for all $x \geq \tau$.
\end{enumerate}
\end{lemma}

Note that the properties asserted in Lemma~\ref{lem:comp} depend only on the parameter $\epsilon$ and the $(1-\epsilon)$-quantile $\tau$ of distribution $P$, and not on the parameter $a$. 

\proof{Proof.} To show property (i), note that on $(\tau,\infty)$, $Q$ has only one atom at $f^{-1}(H)$ at which it achieves the value $1$, and $P$ is an increasing function. Hence, for all $x$, 
\begin{align*}
Q(x)-P(x) 
&\geq P(\tau) - P(f^{-1}(H))\\
&= -\epsilon + (1-P(f^{-1}(H))) \\
&\geq -\epsilon,
\end{align*}
which establishes property (i). 

Property (ii) holds straightforwardly because Algorithm~\ref{alg:disc} transfers any mass of $P$ on $[0,\tau]$ to smaller or equal values. Hence, $Q(x) < P(x)$ can occur only for some $x > \tau$. 

Finally, property (iii) holds because, for all $x \geq \tau$,
\begin{align*}
Q(x) - P(x) 
&\leq Q(f^{-1}(H)) - P(f^{-1}(H))\\
&= 1 - P(f^{-1}(H))\\
&\leq 1 - P(\tau) = \epsilon.
\end{align*}
\endproof

As an example, consider the Pareto distribution $P$, i.e., $P(x) = 1 - (x_m/x)^\beta$ for $x \geq x_m$ and $P(x) = 0$ for $x \leq x_m$, where $x_m > 0$ and $\beta > 1$. Smaller values of $\beta$ correspond to heavier tails of the complementary cumulative distribution function $1-P(x)$. Note that $\beta > 1$ is necessary and sufficient for $\E[X]$ to be finite. It can be readily shown that $\tau = x_m (1/\epsilon)^{1/\beta}$ and $\E[X \mid X > \tau] = (\beta/(\beta-1)) \tau$. Importantly, $\E[X \mid X > \tau]$ can be arbitrarily larger than $\tau$ by taking $\beta$ sufficiently close to $1$. 

\begin{lemma} \label{lem:pareto} 
Assume that $P$ is a Pareto distribution with $x_m > 0$ and $\beta > 1$, and that $f$ satisfies $f(x,0,\ldots,0) = x$ for $x \in \mathbb{R}_+$. Then the following properties hold:
\begin{enumerate}
\item[(i)] For all $x$, $Q(x) - P(x) \geq -(1-(1-1/\beta)^\beta)\epsilon$, which is increasing in $\beta$ on $(1,\infty)$ from $-\epsilon$ to $-(1-1/e)\epsilon$.

\item[(ii)] For all $x \geq \tau$, $Q(x) - P(x) \leq (1-1/\beta)^\beta \epsilon$, which is increasing in $\beta$ on $(1,\infty)$ from $0$ to $(1/e)\epsilon$.

\item[(iii)] For all $x_m \leq x < \tau$, $Q(x) - P(x) \leq \min\{\epsilon / a^\beta, 1\}$.

\item[(iv)] $Q(0) = \max\{1 - \epsilon / a^\beta, 0\}$. In particular, $Q(0) > 0$ if and only if $a \tau > x_m$, i.e., $\epsilon < a^\beta$.
\end{enumerate}
\end{lemma}

\proof{Proof.} Property (i) follows from
\begin{equation*}
Q(x) - P(x) \geq P(\tau) - P(\E[X \mid X > \tau]) = -(1-(1-1/\beta)^\beta)\epsilon.
\end{equation*}

Property (ii) follows from
\begin{equation*}
Q(x) - P(x) \leq 1 - P(\E[X \mid X > \tau]) = (1-1/\beta)^\beta \epsilon \quad \text{for all } x \geq \tau.
\end{equation*}

Property (iii) follows from the observation that for each $x \in [x_m, \tau)$, there exists $1 \leq j \leq J$ such that
\begin{align*}
Q(x) - P(x) &\leq Q(x_j) - P(x_j)\\
&= P(x_{j+1}) - P(x_j)\\
&\leq P(\tilde{x}_{j+1}) - P(\tilde{x}_j),
\end{align*}
where $\tilde{x}_j = a \tau / (1-\epsilon)^{j-1}$ for $1 \leq j \leq J+1$. Under the condition $\epsilon < a^\beta$, 
\begin{align*}
P(\tilde{x}_{j+1}) - P(\tilde{x}_j) &= (\epsilon / a^\beta) (1-\epsilon)^{\beta (j-1)} (1 - (1-\epsilon)^\beta)\\
&\leq \epsilon / a^\beta.
\end{align*}

Finally, property (iv) holds because Algorithm~\ref{alg:disc} assigns $Q(0) = P(a \tau)$. For the Pareto distribution, $P(a \tau) = 1 - \epsilon / a^\beta$ if $a \tau \geq x_m$, and $P(a \tau) = 0$ otherwise. The condition $a \tau \geq x_m$ is equivalent to $\epsilon / a^\beta \leq 1$, yielding
$$
Q(0) = \max\{1 - \epsilon / a^\beta, 0\}.
$$
\endproof

\subsection{Proof of Lemmas for Theorem~\ref{DR}}
\label{app:DR}

\paragraph{Proof of Lemma~\ref{upper-end-f}} 
We first prove the upper bound. Let $T$ be the subset of $S$ containing every $i \in S$ such that $X_i$ exceeds the threshold value $\tau_i$, i.e., $T = \{i \in S \mid X_i > \tau_i\}$.

By Lemma~\ref{lem:sub}, under the condition that $f$ is monotone and either subadditive or submodular, $u$ is a monotone subadditive function. Hence,
\[
\E[f(X_S)] \leq \E[f(X_T)] + \E[f(X_{S\setminus T})].
\]
Noting that
\begin{align*}
\E[f(X_S)] &\leq 2 \max\left\{\E[f(X_T)], \E[f(X_{S\setminus T})]\right\} \\
&\leq 2 \max\left\{\E[f(X_T)], w(S)\right\},
\end{align*}
and by the subadditivity of $u$,
\[
\E[f(X_T)] \leq \E\left[\sum_{i\in T} f(X_i)\right] \leq \epsilon \sum_{i\in S} H_i,
\]
where $H_i = \E[f(X_i)\mid X_i > \tau_i]$ if $P_i(\tau_i) < 1$ and $H_i = 0$ otherwise.  

Thus, the upper bound follows:
\[
\E[f(X_S)] \leq 2 \max\left\{\epsilon \sum_{i\in S} H_i, w(S)\right\}.
\]

Next, we prove the lower bound. Since $f$ is monotone,
\[
\E[f(X_S)] \geq \max\left\{\E[f(X_T)], \E[f(X_{S\setminus T})]\right\},
\]
so that
\begin{equation}
\label{lower-fsDR}
u(S) \geq \max\left\{\E[f(X_T)], w(S)\right\}.
\end{equation}

Now, note that
\begin{align*}
\E[f(X_T)] 
&= \sum_{U \subseteq S} \Pr(T = U) \E[f(X_T) \mid T = U] \\
&\geq \sum_{U \subseteq S: |U|=1} \Pr(T = U) \E[f(X_T) \mid T = U] \\
&= \sum_{i \in S} \Pr(X_i > \tau_i) \Pr(X_j \leq \tau_j, \forall j \neq i) \E[f(X_i) \mid X_i > \tau_i] \\
&\geq (\epsilon - \Delta) (1-\epsilon)^{k-1} \sum_{i\in S} H_i,
\end{align*}
where we used $\Pr(X_j > \tau_j) \geq \epsilon - \Delta$ and $\Pr(X_j \leq \tau_j) \geq 1 - \epsilon$ for all $j \in \Omega$, and the assumption $|S| \leq k$.  

Combining with (\ref{lower-fsDR}), we obtain
\[
u(S) \geq \max\left\{ (\epsilon - \Delta) (1-\epsilon)^{k-1} \sum_{i\in S} H_i, w(S) \right\},
\]
which establishes the lower bound.

\qed

\paragraph{Proof of Lemma~\ref{uv1}} 
By the upper bound in Lemma~\ref{upper-end-f}, we have
\begin{align*}
u(S) &\leq 2 \max\left\{\epsilon \sum_{i\in S} H_i, w(S)\right\} \\
&= \frac{2}{(1-\epsilon)^{k-1}} (1-\epsilon)^{k-1} \max\left\{\epsilon \sum_{i\in S} H_i, w(S)\right\} \\
&\leq \frac{2}{(1-\epsilon)^{k-1}(1-\Delta/\epsilon)} v_1(S),
\end{align*}
where the last inequality uses the lower bound in Lemma~\ref{upper-end-f}.

Similarly, for the lower bound,
\begin{align*}
u(S) &\geq (1-\epsilon)^{k-1} \left(1 - \frac{\Delta}{\epsilon}\right) \max\left\{\epsilon \sum_{i\in S} H_i, w(S)\right\} \\
&= \frac{1}{2} (1-\epsilon)^{k-1} \left(1 - \frac{\Delta}{\epsilon}\right) \cdot 2 \max\left\{\epsilon \sum_{i\in S} H_i, w(S)\right\} \\
&\geq \frac{1}{2} (1-\epsilon)^{k-1} \left(1 - \frac{\Delta}{\epsilon}\right) v_1(S),
\end{align*}
which completes the proof.

\qed

\paragraph{Proof of Lemma~\ref{lem:v2v1ak}} Recall the definition $\tilde{X}_i := \hat{X}_i \ind{\hat{X}_i > a\tau_i}$ and let $\tilde{X}_i^- = \hat{X}_i \ind{\hat{X}_i \leq a\tau_i}$ for $i\in [n]$. For any monotone submodular or monotone subadditive function $f$ and any $a \in [0,1]$,
\begin{align*}
v_1(S) &= \E[f(\hat{X}_S)] \\
&= \E[f(\tilde{X}^-_S + \tilde{X}_S)] \\
&\leq \E[f(\tilde{X}_S^-)] + \E[f(\tilde{X}_S)] \\
&\leq f(a\tau_S) + \E[f(\tilde{X}_S)] \\
&= f(a\tau_S) + v_2(S).
\end{align*}

Using the weak homogeneity of $f$ of degree $d$, we obtain
\[
v_1(S) \leq a^d f(\tau_S) + v_2(S).
\]

Moreover, for any monotone, subadditive or submodular function $f$,
\[
\E[f(\tilde{X}_S)] \geq \frac{\epsilon - \Delta}{k} f(\tau_S).
\]
Indeed, let $j \in \arg\max_{i \in S} \tau_i$. Then,
\begin{align*}
\E[f(\tilde{X}_S)] &\geq \Pr(\tilde{X}_j > \tau_j) f(\tau_j e_j) \\
&\geq (\epsilon - \Delta) f(\tau_j e_j) \\
&\geq \frac{\epsilon - \Delta}{k} f\left(\sum_{i \in S} \tau_j e_i\right) \\
&\geq \frac{\epsilon - \Delta}{k} f\left(\sum_{i \in S} \tau_i e_i\right) \\
&= \frac{\epsilon - \Delta}{k} f(\tau_S),
\end{align*}
where the last inequality follows from (\ref{equ:pi}).

Putting everything together, we have
\[
v_2(S) \geq \frac{1}{1 + a^d k / (\epsilon - \Delta)} v_1(S).
\]

\qed

\paragraph{Proof of Lemma~\ref{middle}} 
By definition of $q$, for every $\tau > 0$ and every $x \in [a\tau, \tau]$,
\[
q(x; \tau, \epsilon, a) \geq (1-\epsilon)x.
\]
Combined with the monotonicity of $f$, this immediately yields
\[
v(S) \geq \E[f((1-\epsilon)\tilde{X}_S)].
\]
Using the weak homogeneity of $f$ with tolerance $\eta$, we conclude
\[
v(S) \geq \frac{1-\epsilon}{\eta} v_2(S).
\]

\qed

\subsection{Proof of Corollary~\ref{cor:cf}}
\label{app:cf}

Note that for all $x \leq 1 - 1/\theta$ and $\theta \geq 1$, we have $1-x \geq e^{-\theta x}$. Consider the case where $\epsilon = c/k$, with $c$ a positive constant in the range $(0,1)$. For this choice of $\epsilon$, we have
\[
(1-\epsilon)^k \geq e^{-\theta c}
\]
provided $c/k \leq 1 - 1/\theta$, which is satisfied for any $\theta \geq 1$.  

Taking $\theta = 1/(1-c)$ ensures that this condition holds, yielding
\[
(1-\epsilon)^k \geq e^{-c/(1-c)}.
\]
This establishes the result stated in the corollary.

\subsection{Approximation results for distributions with arbitrary point masses} 
\label{app:largejump}

For any fixed $\Lambda \in (0,1]$, let $\Omega^* \subseteq \Omega$ denote the set of items whose value distributions have point masses larger than $\Lambda$. For each $i \in \Omega^*$, let $x^*_{i,1}, \ldots, x^*_{i,m_i}$ denote the locations of these point masses, and define $\alpha_{i,j} = P_i(x^*_{i,j}) - \lim_{z \uparrow x^*_{i,j}} P_i(z)$ as the corresponding mass at each point.  

For $i \in \Omega \setminus \Omega^*$, define $P_{i,j}$ and $P_{i,0}$ as follows: $P_{i,j}$ is the point mass distribution concentrated at $x^*_{i,j}$, i.e.,
\[
P_{i,j}(x) = \begin{cases} 
0, & x < x^*_{i,j}, \\
1, & x \geq x^*_{i,j}, 
\end{cases}
\]
and
\[
P_{i,0}(x) = \frac{P_i(x) - \sum_{j=1}^{m_i} \alpha_{i,j} P_{i,j}(x)}{1 - \sum_{j=1}^{m_i} \alpha_{i,j}}
\]
if $\sum_{j=1}^{m_i} \alpha_{i,j} < 1$, and $P_{i,0}(x) = 0$ otherwise.  

With this notation, for every $i \in \Omega^*$, we can decompose
\begin{equation}
P_i(x) = \sum_{j=0}^{m_i} \alpha_{i,j} P_{i,j}(x),
\label{equ:decompose}
\end{equation}
where $\alpha_{i,0} := 1 - \sum_{j=1}^{m_i} \alpha_{i,j}$.

Using (\ref{equ:decompose}), we can express $u(S)$ as a weighted sum of stochastic valuation functions for any $S \subseteq \Omega$. For $S = \{i_1, \ldots, i_k\}$, define virtual items $i_{l,0}, \ldots, i_{l,m_{i_l}}$ for each $i_l \in S$. Then,
\[
u(S) = \sum_{j_1=0}^{m_{i_1}} \cdots \sum_{j_k=0}^{m_{i_k}} \left(\alpha_{i_1,j_1} \cdots \alpha_{i_k,j_k}\right) u(\{i_{1,j_1}, \ldots, i_{k,j_k}\}).
\]

This weighted sum has at most $(1/\Lambda + 1)^k$ terms. Each term $u(\{i_{1,j_1}, \ldots, i_{k,j_k}\})$ can be approximated using the sketch valuation functions: if $j_l = 0$, we discretize $P_{i_l,0}$ using Algorithm~\ref{alg:disc}; otherwise, we use the point mass distribution $P_{i_l,j_l}$.

The results of Theorem~\ref{DR} apply to any item value distributions $P_1, \ldots, P_n$ by choosing $\Delta \in (\epsilon, 1)$ in the decomposition. Similarly, the results of Corollary~\ref{cor:cf} hold by setting $\Lambda = 1/(k h(k))$ with $h(k) > 1$. If $h(k) = \omega(k)$, this yields a constant-factor approximation, with each discretized distribution having support size
\[
s = O\left( \frac{1}{d} k \log k \right).
\]

Consequently, each item is represented by a discretized distribution with support size $O((1/d) k \log k)$ and at most $1/\Delta = k h(k)$ point mass distributions. The overall sketch valuation function obtained from this decomposition has computational complexity
\[
O\big((s k h(k))^k\big).
\]

\subsection{Proof of Theorem~\ref{thm:conc}}

The proof for the upper end is identical to that in Theorem~\ref{DR}, so we focus on the lower end and the middle part.

\paragraph{Lower end.} 
Let $f^*$ denote a concave extension of $f$. Since $f^*(x) = f(x)$ for all $x \in \mathbb{R}_+^n$ and we consider item value distributions with positive support, we may work with
\[
v_1(S) = \E[f^*(\hat{X}_S)] \quad \text{and} \quad v_2(S) = \E[f^*(\tilde{X}_S)].
\]

Recall that
\begin{equation}
\E[f^*(\tilde{X}_S)] \geq \frac{\epsilon-\Delta}{k} \E[f^*(\tau_S)].
\label{equ:recall}
\end{equation}

Define $Z_i = \hat{X}_i - a \tau_i$. Note that $\tilde{X}_i = \hat{X}_i \ind{\hat{X}_i > a \tau_i} \geq Z_i$, and
\[
Z_i = (1-a) \hat{X}_i + a (\hat{X}_i - \tau_i).
\]

Since $f^*$ is monotone, concave, and subadditive, we have
\begin{align*}
v_2(S) 
&\geq \E[f^*(Z_S)] && \text{(monotonicity)}\\
&\geq (1-a) \E[f^*(\hat{X}_S)] + a \E[f^*(\hat{X}_S - \tau_S)] && \text{(concavity)}\\
&\geq (1-a) \E[f^*(\hat{X}_S)] + a \E[f^*(\hat{X}_S)] - a f^*(\tau_S)) && \text{(subadditivity)}\\
&\geq v_1(S) - \frac{a k}{\epsilon-\Delta} v_2(S) && \text{(using \eqref{equ:recall}).}
\end{align*}

\paragraph{Middle part.} 
The middle part follows exactly as in Theorem~\ref{DR}, noting that any concave function is weakly homogeneous with tolerance $1$.  

\qed

\subsection{Proof of Theorem~\ref{thm:coordinate}}
\label{app:coordinate}

Again, the upper end follows exactly as in Theorem~\ref{DR}. We focus on the lower end and middle parts.

\paragraph{Lower end.} 
The following lemma provides the required bound.

\begin{lemma}
Assume that $f$ is monotone, either subadditive or submodular, and coordinate-wise weakly homogeneous of degree $d$ over $[0,1]$. Then,
\[
v_2(S) \geq \frac{1}{1 + a^d k / (\epsilon - \Delta)} v_1(S).
\]
\end{lemma}

\noindent
\textit{Proof.}  
The argument is analogous to the proof of Lemma~\ref{lem:v2v1ak}, using the fact that under coordinate-wise weakly homogeneous conditions,
\[
f(a \tau_S) \leq a^{d k} f(\tau_S) \leq a^d f(\tau_S).
\]  
\qed

\paragraph{Middle part.} 
As in Theorem~\ref{DR}, the middle part follows from repeated application of the weak homogeneity property coordinate-wise. This yields
\[
\E[f((1-\epsilon)\tilde{X}_S)] \geq (1/\eta)^k (1-\epsilon)^k \E[f(\tilde{X}_S)].
\]

\qed

\begin{figure}[t]
\centering
\includegraphics[width=0.75\linewidth]{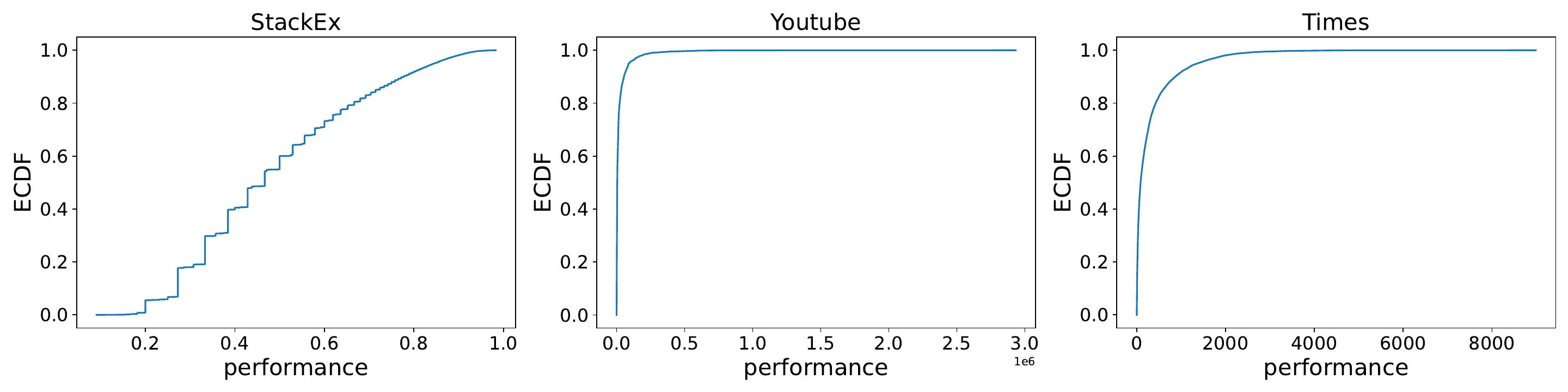}
\caption{Empirical CDFs of performance values for the three datasets.}
\label{fig:ecdf}
\end{figure}

\section{Supplementary Results for Real-World Data Experiments}  \label{sec:real}
		
\subsection{Information about Datasets}

\paragraph{StackExchange Dataset.} 
This dataset contains 35,218 questions and 88,584 answers on Academia.StackExchange, retrieved from the official StackExchange data dump on Jan 20, 2022 (\url{https://archive.org/details/stackexchange}). Each answer receives up-votes and down-votes indicating quality. We considered only users who submitted at least 100 answers. For an answer $a$ to question $q$ with $u(a,q)$ up-votes and $d(a,q)$ down-votes, we define a quality metric
\[
s(a,q) = \frac{u(a,q) + c_1}{u(a,q) + d(a,q) + c_2},
\]
where $c_1, c_2 > 0$. This metric is motivated by Bayesian estimation and used in \cite{SVY21}. The ratio increases with up-votes and decreases with down-votes. It is called \emph{balanced} when $c_1/c_2 = 1/2$ and \emph{conservative} if $c_1/c_2 < 1/2$. In our main experiment, we chose the conservative setting $(c_1, c_2) = (2,8)$. Results for other values are in Section~\ref{sec:stackex}.

\paragraph{YouTube Dataset.} 
The YouTube dataset~\cite{kaggle.com2021} contains information on 37,422 unique videos, including publication dates, view counts, likes, and dislikes, for the period August 2020 to December 2021 in the USA, Canada, and Great Britain. For our experiments, we filtered out YouTubers with fewer than 50 uploads. In the main experiment, we used the daily view counts as the performance measure. Other metrics were also tested (see Section~\ref{sec:youtube}).

\paragraph{New York Times Dataset.} 
This dataset~\cite{kaggle.com2020} contains 16,570 articles and associated comments from January to December 2020. Each article belongs to a section. For our experiments, we used all articles and their comment counts as performance measures.

\paragraph{Empirical Distributions.} 
Figure~\ref{fig:ecdf} shows the empirical CDFs of performance values aggregated over all data points for the three datasets. The distributions differ substantially, indicating that our method works across a wide range of input value distributions.

\subsection{StackExchange Dataset: Other Parameter Settings} \label{sec:stackex}

We tested $(c_1,c_2)$ pairs: $(2, 8)$, $(8, 32)$, and $(10, 10)$, corresponding to conservative (first two) and balanced (last) settings. Figure~\ref{fig:ecdf_Stackex} shows empirical CDFs for these settings. Larger values of $c_1$ and $c_2$ lead to stronger concentration of the performance values around $c_1/c_2$. Figure~\ref{fig:box_Stackex} shows corresponding approximation ratios. All ratios remain highly concentrated around 1, indicating that the choice of $(c_1,c_2)$ has little effect on the approximation ratios.

\begin{figure}[t]
\centering
\includegraphics[width=0.75\linewidth]{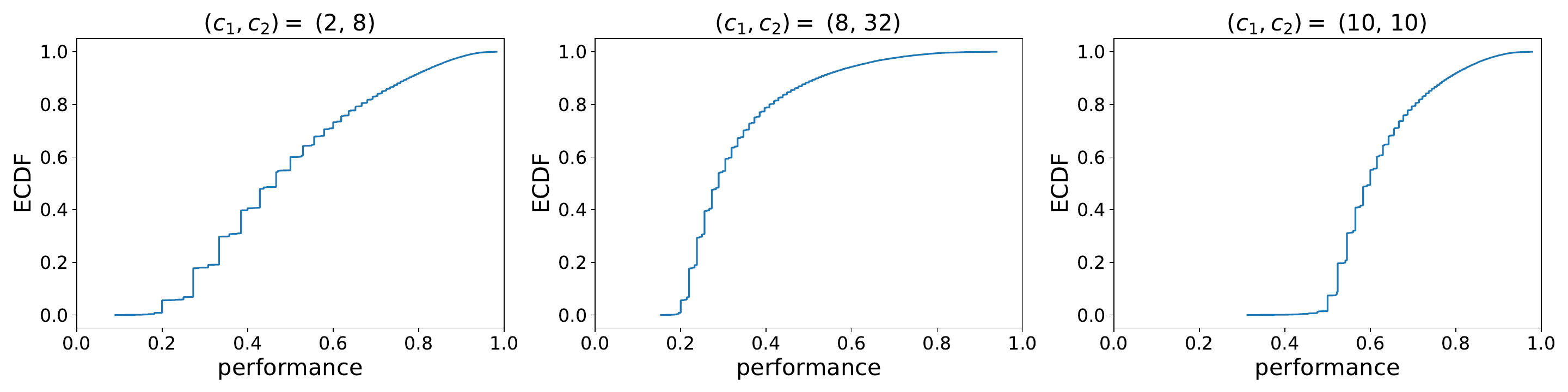}
\caption{Empirical CDFs of performance values for the StackExchange dataset under three parameter settings.}
\label{fig:ecdf_Stackex}
\end{figure}

\begin{figure}[t]
\centering
\begin{tabular}{c}
{\footnotesize (a) Set size $k=5$}\\
\includegraphics[width=0.75\linewidth]{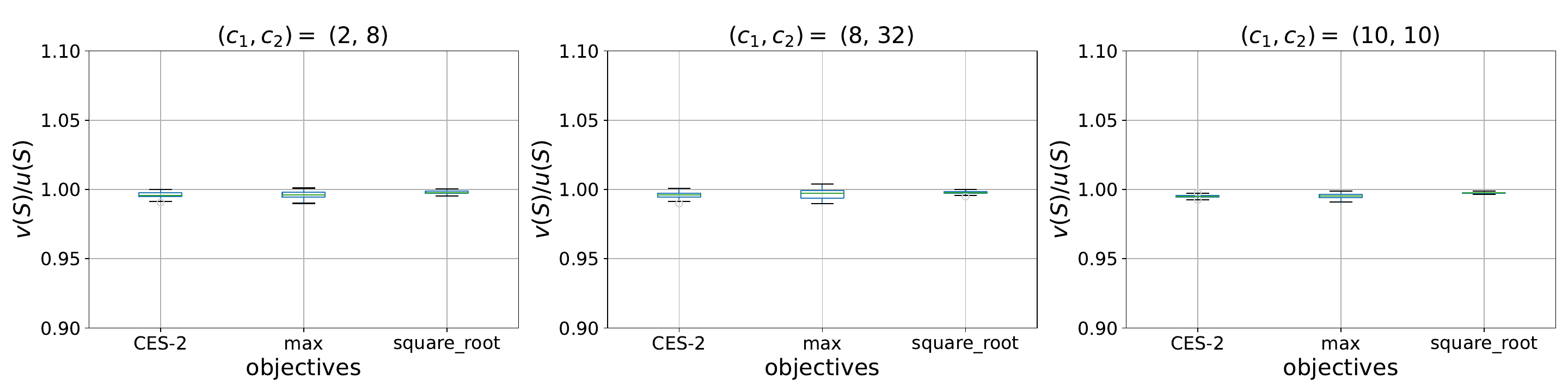}\\
{\footnotesize (b) Set size $k=10$}\\
\includegraphics[width=0.75\linewidth]{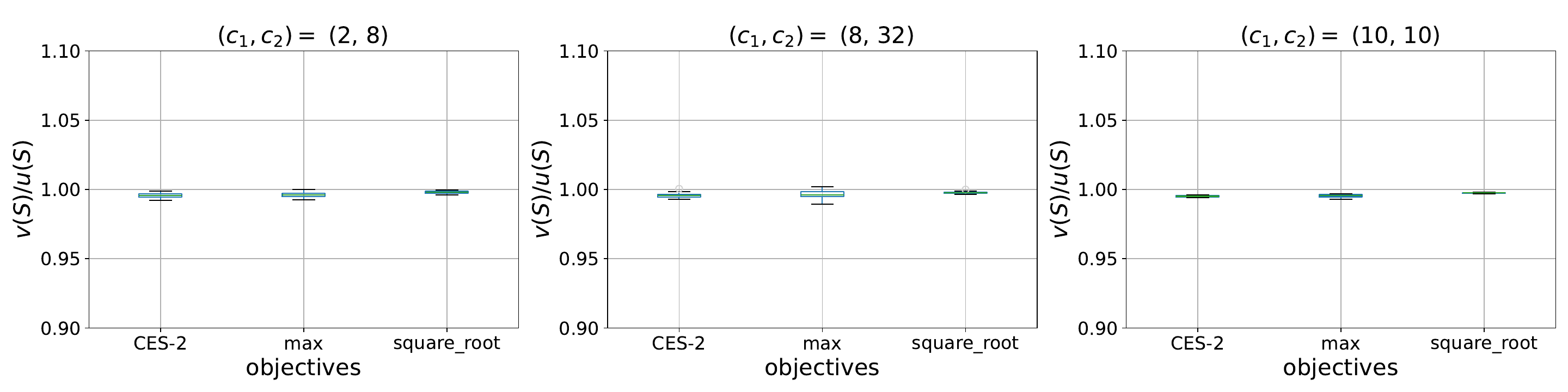}
\end{tabular}

\caption{Approximation ratios for different valuation functions on the StackExchange dataset, under different $(c_1,c_2)$ parameter settings.}
\label{fig:box_Stackex}
\end{figure}

\subsection{YouTube Dataset: Other Performance Metrics}  \label{sec:youtube}

We evaluate the function approximation accuracy for six different performance measures of a video content item. Consider a video that has been uploaded for {\tt \#lifetime} days and has accumulated {\tt \#views} views, {\tt \#likes} likes, and {\tt \#dislikes} dislikes over that period. The performance measures we consider are defined as follows.

The \emph{views per day} is the number of views per day, i.e., {\tt (\#views)/(\#lifetime)}. The \emph{log view count} is defined as {\tt log(\#views + 1)}. The \emph{standard like ratio} is defined as {\tt (\#likes)/(\#likes + \#dislikes)}. The \emph{video power index (VPI)} is defined as the product of the view ratio and the standard like ratio.

The \emph{Bayesian like ratio} is defined as
{\tt (\#likes + $c_1$)/(\#likes + \#dislikes + $c_2$)},
where $c_1$ and $c_2$ are positive parameters. We consider two parameter settings: a conservative case with {\tt $c_1 = 0.01 \times$ \#views} and {\tt $c_2 = 0.1 \times$ \#views}, and a balanced case with {\tt $c_1 = 0.05 \times$ \#views} and {\tt $c_2 = 0.1 \times$ \#views}. The Bayesian like ratio adjusts for low view counts and is conceptually analogous to the StackExchange scoring metric.


Figure~\ref{fig:ecdf_Youtube} shows the empirical CDFs of all six measures. These distributions differ in shape and the range of the support.  
 
\begin{figure}[t]
\centering
\includegraphics[width=0.75\linewidth]{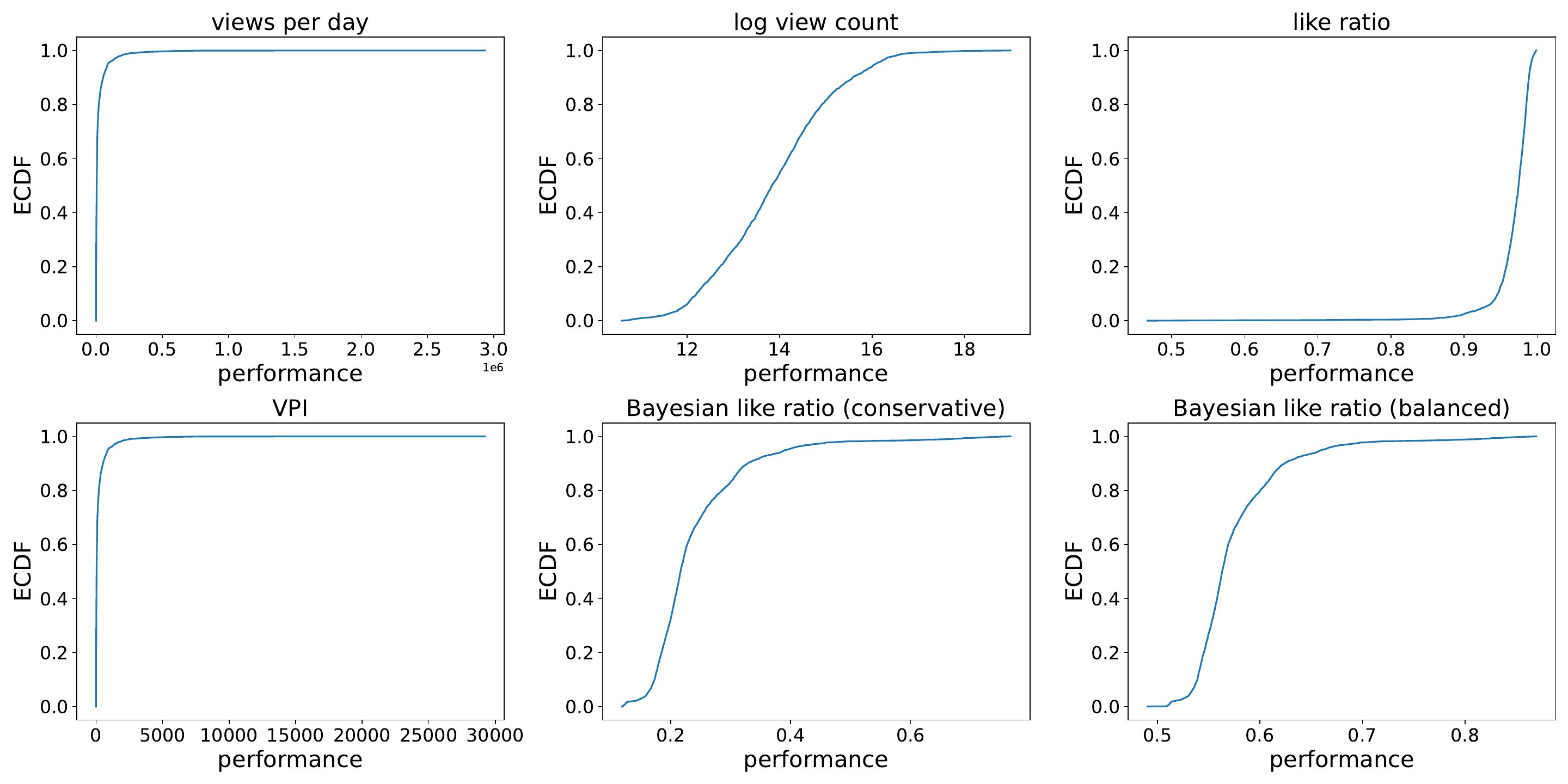}
\caption{Empirical CDFs of six performance metrics for the YouTube dataset.}
\label{fig:ecdf_Youtube}
\end{figure}

Figure~\ref{fig:box_Youtube} presents approximation ratios for three objective functions (max, CES of degree 2, and square root of sum). Ratios for the log view count, standard like ratio, and balanced Bayesian ratio are tightly concentrated around 1, while the view ratio and VPI show greater variability, consistent with the heavier tails of the corresponding value distributions.

\begin{figure}[t]
\centering
\begin{tabular}{c}
{\footnotesize (a) Set size $k=5$}\\
\includegraphics[width=0.75\linewidth]{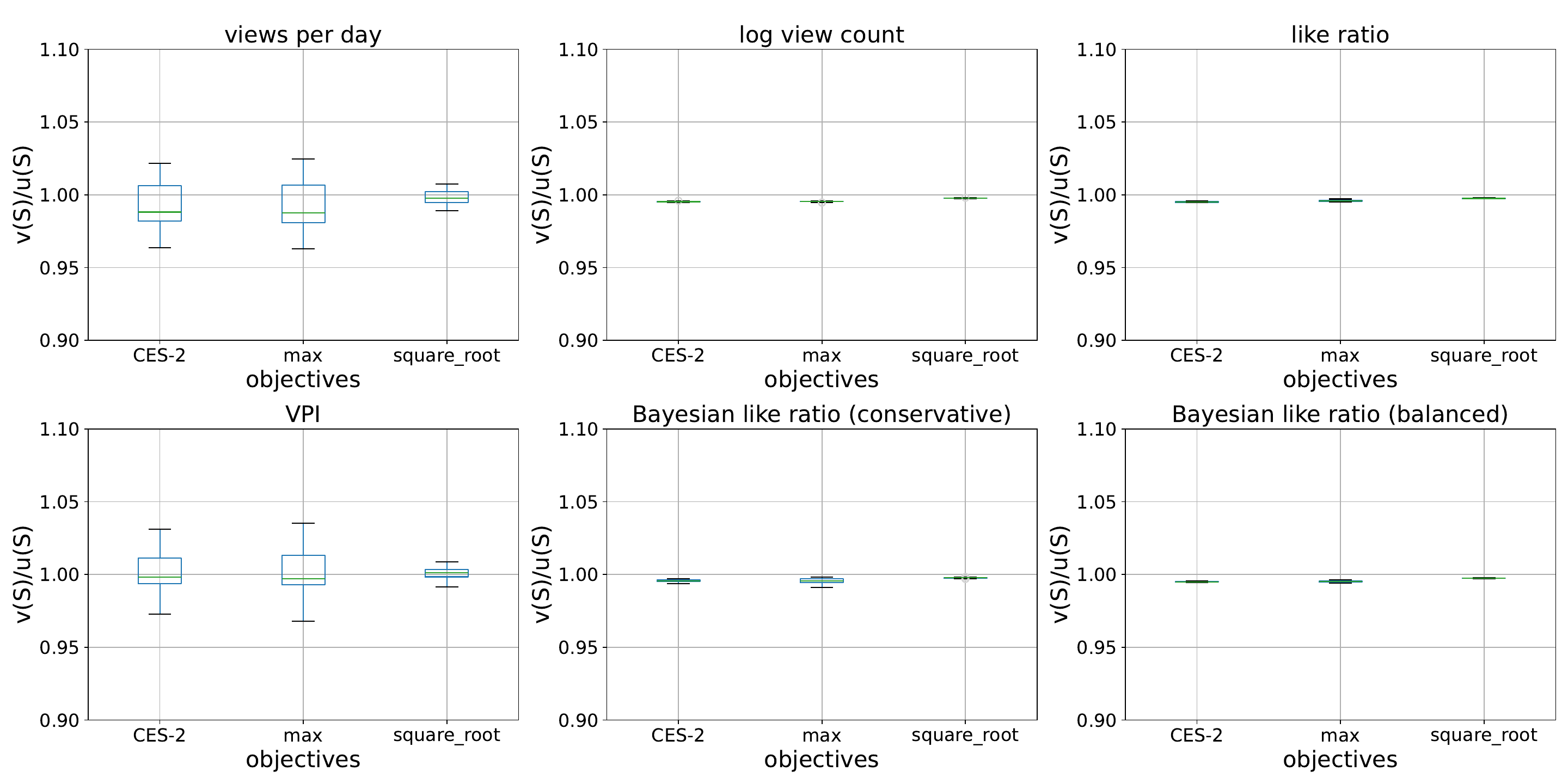}\\
{\footnotesize (b) Set size $k=10$}\\
\includegraphics[width=0.75\linewidth]{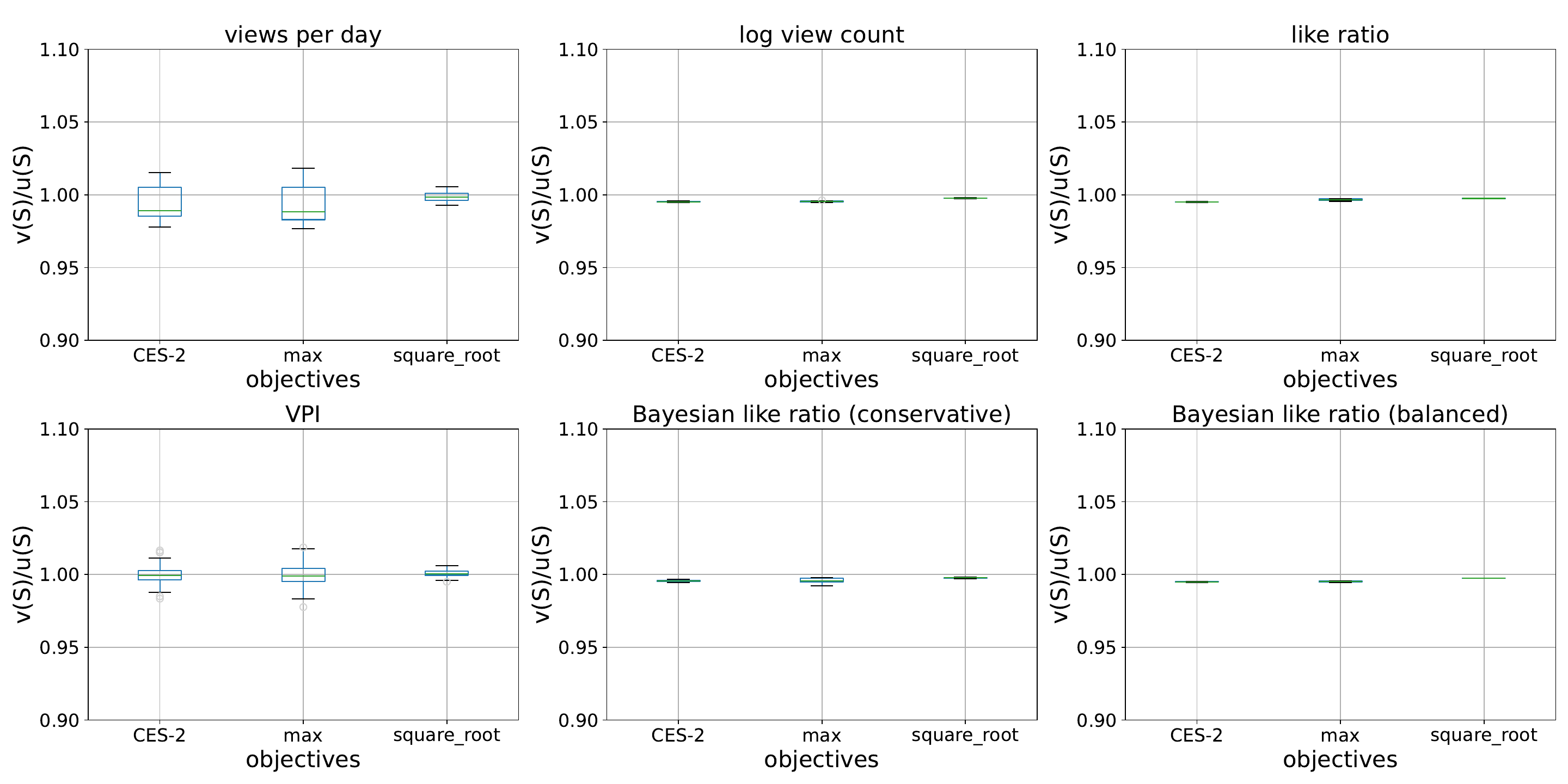}
\end{tabular}

\caption{Approximation ratios for different valuation functions on the YouTube dataset, under six performance metrics.}
\label{fig:box_Youtube}
\end{figure}


\end{document}